\documentclass[11pt]{amsart}
\usepackage{amsmath}
\usepackage{amssymb}
\usepackage{epsfig}
\usepackage{graphics}
\usepackage{subcaption}
\usepackage{amsthm}
\usepackage{hyperref}
\usepackage{multicol}
\usepackage{bm}
\usepackage{url}

\usepackage{color}

\numberwithin{equation}{section}

\usepackage{cleveref}

\newtheorem{theorem}{Theorem}
\newtheorem{lemma}{Lemma}[section]
\newtheorem{corollary}{Corollary}

\newtheorem{proposition}{Proposition}
\newtheorem{remark}{Remark}[section]

\newcommand\be{\begin{equation}}
\newcommand\ee{\end{equation}}
\newcommand\bea{\begin{eqnarray}}
\newcommand\eea{\end{eqnarray}}
\newcommand\beaa{\begin{eqnarray*}}
\newcommand\eeaa{\end{eqnarray*}}
\newcommand\bay{\begin{array}}
\newcommand\eay{\end{array}}
\newcommand\ba{\begin{align}}
\newcommand\ea{\end{align}}

\newcommand\dps{\displaystyle}

\usepackage{graphicx}
\ifx\pdfoutput\undefined
\else
\fi
\graphicspath{{images/}}

\usepackage[inline]{enumitem}

\newcommand{\tr}{\textup{tr}}

\newcommand{\mcI}{\mathcal{I}}

\title[]{Extinction and the global structure of positive equilibria in a ratio-dependent intraguild predation model}
\author[T.-H. Chan]{Tzu-Hsuan Chan}
\address{Department of Applied Mathematics, National Yang Ming Chiao Tung University, Hsinchu 300, Taiwan.}
\email{iljanjc@gmail.com}
\author[C.-H. Wu]{Chang-Hong Wu}
\address{Department of Applied Mathematics, National Yang Ming Chiao Tung University, Hsinchu 300, Taiwan.}
\email{changhong@math.nctu.edu.tw}

\begin{document}


\maketitle

\setlength{\baselineskip}{13pt}
\begin{abstract} 
We study an intraguild predation model with ratio-dependent trophic interactions. We first derive sufficient conditions for consumer extinction, total extinction, and global convergence to biologically relevant boundary equilibria. We then turn to coexistence at equilibrium and introduce a biologically motivated coordinate transformation that reduces the positive-equilibrium problem to an explicitly parameterized one-dimensional curve. This reduction enables us to determine how the positive-equilibrium branch connects to boundary equilibria and to classify all of its admissible global geometries.
We also discuss the stability of positive equilibria. 
Numerical examples suggest the presence of periodic and chaotic coexistence, as well as several forms of multiple attractors.
\end{abstract}

\section{Introduction}
In a three-species ecological system, the interaction structures, such as food chains and competing predators, have been classified into 34 cases \cite{krikorian1979volterra}. Among them, five main structures have been extensively discussed in \cite{hsu2015analysis}. In particular, intraguild predation (IGP, also called omnivory) is a three-species structure: the resource, the IG prey, and the IG predator. In particular, IG prey and IG predator are called the consumers. In this structure, the IG prey has only one predation pathway to the resource, but the IG predator has two predation pathways to the resource and the IG prey. This makes the IGP structure asymmetric and blends several simple ecological structures, including the simple food chain, the exploitative competition, and the apparent competition \cite{polis1989ecology,polis1992intraguild,holt1997theoretical}.

The general IGP model can be written as (\cite{holt1997theoretical})
\begin{equation}\label{eq:general-IGP}
\begin{aligned}
    x'&=x\phi(x)-yg_1(x,y,z)-zg_2(x,y,z)\\
    y'&=e_1yg_1(x,y,z)-zg_3(y,z)-d_1y\\
    z'&=e_2zg_2(x,y,z)+e_3zg_3(y,z)-d_2z
\end{aligned}
\end{equation}
where $x,y,z$ represent the resource, IG prey, and IG predator; $g_1$, $g_2$, and $g_3$ are trophic functions; $\phi$ is the recruitment function of the resource; $e_1,e_2,e_3$ are the efficiency of the consumer converting consumed prey; and $d_1,d_2$ are the death rate of IG prey and IG predator. 

Several studies have investigated the dynamics and coexistence state of IGP models. For general trophic function, Holt and Polis \cite{holt1997theoretical} provide that the coexistence only occurs when the IG prey is superior in the exploitative competition. Bl\'{e} et al.  \cite{ble2022stable} established Hopf conditions for stable coexistence cycles and showed that Bautin bifurcation can generate bistability in IGP models with general prey-dependent trophic function. For linear functional responses, Tanabe and Namba \cite{tanabe2005omnivory} demonstrated that even simple linear functional responses can generate chaotic dynamics in a three-species IGP model. Hsu et al. \cite{hsu2015analysis} completely classified the parameter space into extinction and uniform-persistence regimes for a specialist-predator formulation. Yang et al. \cite{yang2016global} classified all local dynamics and most of the global dynamics for a generalist-predator formulation. For nonlinear trophic function, Kang and Wedekin \cite{kang2013dynamics} compared IGP models with generalist and specialist predators, assuming Holling type I functional responses between the basal resource and both consumers and a Holling type III response between the intraguild prey and predator, and found that the generalist-predator model is prone to coexist. Jin et al. \cite{jin2025exploring} classified the coexistence equilibria of a Holling type II IGP model and established transcritical and Hopf bifurcations governing transitions between stable and oscillatory coexistence. Ji and Wang \cite{ji2022competitive} showed that predator interference can promote and stabilize coexistence while also generating bistability and complex dynamics in the IGP model with a Beddington-DeAngelis-type trophic function.

In reality, as predator density increases, the amount of prey available to each predator may decrease due to the limit of prey, giving rise to a prey-sharing effect \cite{arditi1989coupling,tyutyunov2021ratio}. Consequently, Arditi and Ginzburg \cite{arditi1989coupling} proposed a ratio-dependent formulation to explicitly account for prey sharing among predators. Some biological research also supports the ratio-dependent trophic function in ecosystems \cite{arditi1991variation,ginzburg1992consequences,sentis2013functional,hossie2016spatial}, and some mathematical research has studied models with ratio-dependent trophic function \cite{kuang1998global,hsu2001rich,hsu2003ratio} across different ecological structures.

The ratio-dependent trophic functions appearing in \eqref{eq:general-IGP} have been investigated in previous studies \cite{arino2004infection,freeze2014analysis}. However, the resulting coexistence conditions are often complex, making it difficult to classify the global structure of positive equilibria. Our objective is to formulate a three-species ratio-dependent IGP model and identify biologically meaningful thresholds that provide a clear classification of its global topological structure of positive equilibria.

Therefore, we propose the following three-species IGP model with ratio-dependent trophic function:
\begin{equation}\label{eq:main-model-no-rescale}
\begin{aligned}
    &x'=rx\left(1-\frac{x}{K}\right)-\frac{c_1xmy}{a_1x+my+z}-\frac{c_2xz}{a_2x+my+z}\\
    &y'=\frac{e_1c_1xmy}{a_1x+my+z}-\frac{c_3yz}{a_3y+z}-d_1y\\
    &z'=\frac{e_2c_2xz}{a_2x+my+z}+\frac{e_3c_3yz}{a_3y+z}-d_2z\\
    &x(0)=x_0\geq0,\quad y(0)=y_0\geq0,\quad z(0)=z_0\geq0
\end{aligned}
\end{equation}
where $r$ represents the intrinsic growth rate and $K$ stands for the carrying capacity of the resource;
$c_i$ is the searching efficiency rate; $m$ is the relative predation efficiency between IG prey and IG predator; $e_i$ is the conversion constant; $a_i$ is the saturation parameter; $d_j$ is the death rate. For biological meaning, we assume $r,K,c_i,e_i,a_i,d_j,m$ to be positive for $i=1,2,3$ and $j=1,2$.

Our work focuses on two issues: extinction and coexistence, with particular emphasis on the organization of positive equilibria. The extinction analysis consists of two parts: 
\Cref{thm:main-extinction} provides sufficient conditions for IG prey or IG predator extinction, and the total extinction, while \Cref{thm:competition-exclusion} 
establishes competitive-exclusion results.

The main difficulty in analyzing positive equilibria is that the three equilibrium equations form a strongly coupled system of rational equations and do not admit a useful explicit solution in the original variables.
To overcome this difficulty, we introduce the biologically motivated change of variables \eqref{eq:change-of-coordinate}, which
separates the resource abundance, the relative contributions of the two
consumers to their total weighted abundance, and the ratio of this
total abundance to the resource.
Under this transformation, the positive-equilibrium problem reduces,
under the generic assumptions, to a one-dimensional curve. 
This reduction allows us to track how positive equilibria vary and determine whether and how the positive-equilibrium branch connects to boundary equilibria, and classify all of its admissible global geometries. 
We first establish the transcritical bifurcation at $E_{xz}$ (see \Cref{tab:boundary-equilibria}) in \Cref{thm:bifurcation-Exz}. We then classify the global structure of the positive-equilibrium branch in
\Cref{thm:structure-branch}.  Its local stability is studied in
\Cref{prop:las-positive} and \Cref{thm:for-m-large}, while
\Cref{thm:saddle-node} identifies conditions under which the branch
undergoes a saddle-node bifurcation.  Finally, numerical examples illustrate nonequilibrium coexistence dynamics, including periodic oscillations and chaotic behavior.



An important parameter of our model is 
the parameter $m$, referred to as the relative predation efficiency \cite{yang2002basic}. 
In \cite{hsu2001rich}, the authors studied a competition model with ratio-dependent structure. The relative predation efficiency parameter was normalized and, therefore, was not treated as an explicit parameter determining the dynamics. In our study, we will retain $m$ explicitly and treat it as a bifurcation parameter, allowing us to examine how changes in relative predation efficiency reorganize the existence and stability of  positive equilibria.


The remainder of this paper is organized as follows. In \Cref{sec:basic-prop}, we establish some basic properties and important quantities of the model. In \Cref{sec:extinction}, we derive several sufficient conditions for different extinction scenarios. In \Cref{sec:competition-exclusion}, we apply the theory of limiting systems to provide the competition exclusion phenomena. In \Cref{sec:coexistence}, we apply bifurcation theory and a change of coordinates to justify the existence and stability of coexistence states. In \Cref{sec:numerical}, we present several numerical results, including multiple attractors, Hopf bifurcation, and chaos.

\section{Preliminary}\label{sec:basic-prop}

It is more convenient for us to rescale some variables in \eqref{eq:main-model-no-rescale}. We first rescale population densities and time:
\begin{align*}
  \overline{x}=\frac{x}{K}, \quad   \overline{y}=\frac{y}{K},\quad \overline{z}=\frac{z}{K},\quad \tau=r t.
\end{align*}
Also, let 
$\overline{a_i}=a_i$, 
$\overline{c_i}=c_i/r$, 
$\overline{e_i}=e_i$, $i=1,2,3$;
$\overline{d_j}=d_j/r$, $j=1,2$.
After changing those variables, we drop the bars from the equations, and we will still use $t$ to represent the time, then we will get the reduced system
\begin{align}
    &x'=x(1-x)-\frac{c_1mxy}{a_1x+my+z}-\frac{c_2xz}{a_2x+my+z}=:X(x,y,z)\label{eq:main-model-1}\\
    &y'=\frac{e_1c_1mxy}{a_1x+my+z}-\frac{c_3yz}{a_3y+z}-d_1y=:Y(x,y,z)\label{eq:main-model-2}\\
    &z'=\frac{e_2c_2xz}{a_2x+my+z}+\frac{e_3c_3yz}{a_3y+z}-d_2z=:Z(x,y,z)\label{eq:main-model-3}
\end{align}
with initial condition
\begin{align}\label{eq:main-model-initial}
    x(0)=x_0\geq0,\quad y(0)=y_0\geq0,\quad z(0)=z_0\geq0.
\end{align}
Here $m$ represents the relative predation efficiency. We explicitly retain this parameter because it is important to our results.

\subsection{Well-posedness}

Set ${\bm F}(x,y,z):=(X(x,y,z),Y(x,y,z),Z(x,y,z))$ and
\begin{align*}
\mathbb{R}^3_{>0}:=\{(x,y,z)\in\mathbb{R}^3:x,y,z>0\}, \quad  
\mathbb{R}^3_{\geq0}=\{(x,y,z)\in\mathbb{R}^3:x,y,z\ge0\}. 
\end{align*}
Note that the vector field ${\bm F}$ is well-defined in $\mathbb{R}^3_{>0}$, but some ratio-dependence terms need extension on the boundary.
Note that these ratio-dependence terms have a continuous and locally Lipschitz extension by assigning the value zero on their singular set. Consequently, the vector field ${\bm F}$ can be viewed as a continuous and locally Lipschitz function defined on $\mathbb{R}^3_{\geq0}$. Therefore, by the theory of existence and uniqueness for ODEs, the solution of \eqref{eq:main-model-1}--\eqref{eq:main-model-3} with initial conditions \eqref{eq:main-model-initial} exists locally in time uniquely.

To ensure biological meaning, we provide positivity and boundedness of the solution under nonnegative initial conditions. First, note that
\[
X(0,y,z)=0,\quad
Y(x,0,z)=0,\quad
Z(x,y,0)=0,
\]
Thus, each coordinate plane is invariant and  $\mathbb{R}_{\geq0}^3$ is positively invariant. 
Write 
\[
x'=xF_1(x,y,z),\quad
y'=yF_2(x,y,z),\quad
z'=zF_3(x,y,z),
\]
together with the local boundedness of $F_i$, which yields that $\mathbb{R}_{>0}^3$ is positively invariant.

We now establish boundedness. Along any nonnegative solution,
\[
x'\leq x(1-x),\quad
y'\leq e_1c_1x-d_1y,\quad
z'\leq e_2c_2x+e_3c_3y-d_2z.
\]
Set
\[
M_x:=\max\{x_0,1\},~
M_y:=\max\left\{y_0,\frac{e_1c_1M_x}{d_1}\right\},~
M_z:=\max\Big\{
z_0,\frac{e_2c_2M_x+e_3c_3M_y}{d_2}
\Big\}.
\]
By comparison, we have
\[
0\leq x(t)\leq M_x,\qquad
0\leq y(t)\leq M_y,\qquad
0\leq z(t)\leq M_z
\]
as long as the solution exists. Therefore, the a priori bounds show that the solution exists globally.


In summary, we have the following result. 

\begin{proposition}\label{prop:well-posedness}
The system \eqref{eq:main-model-1}--\eqref{eq:main-model-3} with initial data \eqref{eq:main-model-initial} admits a unique global nonnegative solution. Moreover, $\mathbb{R}_{\geq0}^3$ and $\mathbb{R}_{>0}^3$ are positively
invariant, and each solution is bounded. 
More precisely, 
$\Omega_1=\{(x,y,z)|\, 0\leq x \leq M_x,\  0\leq y \leq M_y,\  0\leq  z \leq M_z \}$
is positively invariant.
\end{proposition}

\subsection{Boundary Equilibria}

In our model \eqref{eq:main-model-1}--\eqref{eq:main-model-3}, it is easy to see that there are four possible boundary equilibria (under continuous extension); we list them and the corresponding existence conditions in \Cref{tab:boundary-equilibria}. 
Note that $E_0=(0,0,0)$ is referred to as 
the total extinction equilibrium. 
It is a degenerate extinction equilibrium since the ratio-dependent functional responses are singular at the total-extinction state before the continuous extension of the vector field is considered.
In addition,  $E_1=(1,0,0)$ is also degenerate, and is called the consumer-free equilibrium or resource-only equilibrium; $E_{xy}=(x_1,y_1,0)$
is called the IG predator-free equilibrium; 
$E_{xz}=(x_2,0,z_2)$ is called the IG prey-free equilibrium.

\begin{table}
    \centering
    \begin{tabular}{|c|c|}\hline
        Boundary equilibria & Conditions \\\hline\hline
        $E_0=(0,0,0)$ & Always\\
        $E_1=(1,0,0)$ & Always\\
        $E_{xy}=(x_1,y_1,0)=(1-\frac{e_1c_1m-a_1d_1}{e_1m},\frac{e_1c_1m-a_1d_1}{d_1m}x_1,0)$ & $0<e_1c_1m-a_1d_1<e_1m$\\
        $E_{xz}=(x_2,0,z_2)=(1-\frac{e_2c_2-a_2d_2}{e_2},0,\frac{e_2c_2-a_2d_2}{d_2}x_2)$ & $0<e_2c_2-a_2d_2<e_2$\\\hline
    \end{tabular}
    \caption{Boundary equilibria for \eqref{eq:main-model-1}--\eqref{eq:main-model-3}.}
    \label{tab:boundary-equilibria}
\end{table}

\subsection{Threshold values}

In this subsection, we introduce threshold values
to determine the dynamics of \eqref{eq:main-model-1}--\eqref{eq:main-model-3}. The biological meanings are also mentioned. 

We first introduce the invasion rates $\mcI_y^1$. Intuitively, they come from the per-capita growth rates of rare invaders at the resource-only equilibrium (environment) $E_1$. To see this,
from \eqref{eq:main-model-2}, we 
can compute the invasion rate of $y$
\begin{align}\label{I-y}
\mcI_y^1:=\lim\limits_{\substack{y\to0^+\\(x,z)=(1,0)}}\frac{y'}{y}=\frac{e_1c_1m}{a_1}-d_1,
\end{align}
which can be interpreted as the invasion rate of the IG prey $y$ at the
resource-only equilibrium $E_1=(1,0,0)$. Biologically, $\mcI_y^1$ is the
per-capita net growth rate of a rare IG prey population when the resource is near its carrying capacity, and the IG predator is absent. Consequently, $\mcI_y^1>0$ (resp., $\mcI_y^1<0$) means that $y$ may (resp., may not) be able to invade the resource-only state. The other invasion rates can be computed in a similar argument and are listed in \Cref{tab:invasion-rate}.

\begin{table}
    \centering
    \renewcommand{\arraystretch}{1.3}
    \begin{tabular}{|c|c|c|}\hline
        consumer & environment & invasion rate \\\hline\hline
        $y$ & $E_1$ & $\mcI_y^1=\frac{e_1c_1m}{a_1}-d_1$\\
        $z$ & $E_1$ & $\mcI_z^1=\frac{e_2c_2}{a_2}-d_2$\\
        $z$ & $y$ only & $\mcI_z^y=\frac{e_3c_3}{a_3}-d_2$\\\hline 
        $y$ & $E_{xz}$ & $\mcI_y^{xz}=\frac{e_1c_1md_2}{e_2c_2+(a_1-a_2)d_2}-(c_3+d_1)$\\
        $z$ & $E_{xy}$ & $\mcI_z^{xy}=\frac{e_2c_2d_1}{e_1c_1m+(a_2-a_1)d_1}+\mcI_z^y$\\\hline
    \end{tabular}
    \caption{The table of invasion rate. Each row shows that different consumers invade different environments.}
    \label{tab:invasion-rate}
\end{table}

There is an important invasion rate $\mcI_z^y$. This quantity is calculated by
\begin{align*}
    \mcI_z^y=\lim\limits_{\substack{z\to0^+\\x=0,~y>0}}\frac{z'}{z}=\frac{e_3c_3}{a_3}-d_2,
\end{align*}
which represent the per-capita net growth rate of a rare IG predator population at the IG prey-only environment. If $\mcI_z^y>0$, then the IG predator can initially increase under this limiting feeding condition; it does not by itself imply long-term persistence in the absence of the resource; while $\mcI_z^y<0$ does not imply the IG predator will go extinct due to the predation pathway from the IG predator to the resource.

\begin{remark}
The invasion rates $\mathcal{I}_{y}^{1}$ and $\mathcal{I}_{y}^{xz}$ are increasing in $m$; while $\mcI_z^1$ and $\mcI_z^y$ are independent of $m$ and 
$\mathcal{I}_{z}^{xy}$ is decreasing in $m$.
This means that increasing $m$ may raise the per-capita gain of the IG prey from
the resource and facilitate its invasion of $E_{xz}$, but reduces the resource contribution to the growth of the IG predator at $E_{xy}$. This suggests that the effect of $m$ on coexistence is not simply monotone: increasing $m$ neither uniformly promotes nor uniformly suppresses coexistence.
\end{remark}

\section{Extinction}\label{sec:extinction}

In this section, we will provide some sufficient conditions for the extinction of the three species, and the conclusion will be shown in \Cref{thm:main-extinction}.


\subsection{Limiting Systems}\label{subsec:limiting}

We will use the limiting system in the further subsection, so we first discuss several properties of the limiting systems when $y\to0$ and $z\to 0$ as $t\to\infty$. The two limiting systems can be written as 
\begin{center}
\begin{minipage}{.45\textwidth}
\begin{equation}\label{eq:limiting-xy}
\begin{cases}
 \dps  x'=x(1-x)-\frac{c_1mxy}{a_1x+my}\\
  \dps   y'=\frac{e_1c_1mxy}{a_1x+my}-d_1y
\end{cases}
\end{equation}
\end{minipage}\hfill
\begin{minipage}{.45\textwidth}
\begin{equation}\label{eq:limiting-xz}
\begin{cases}
 \dps   x'=x(1-x)-\frac{c_2xz}{a_2x+z}\\
  \dps  z'=\frac{e_2c_2xz}{a_2x+z}-d_2z
\end{cases}  
\end{equation}
\end{minipage}
\end{center}
Since \eqref{eq:limiting-xy} and \eqref{eq:limiting-xz} are limiting systems of the system \eqref{eq:main-model-1}--\eqref{eq:main-model-3}, positivity and boundedness will be preserved. Since these two limiting systems have the same structure, we only provide the proof for \eqref{eq:limiting-xy}, and the properties for \eqref{eq:limiting-xz} can be proved by a similar argument.

Clearly, there are only three possible equilibria for \eqref{eq:limiting-xy}: the equilibria $(0,0)$ and $(1,0)$ always exist, and $(x_1,y_1)$ exist if and only if $0<\mcI_y^1<\frac{e_1m}{a_1}$.

\begin{lemma}\label{lemma:E0-repeller-limiting}~
    \begin{enumerate}
        \item If $c_1<1$, then $(0,0)$ is a uniformly strong repeller for $\{(x,y):x>0,y\geq0\}$ for the limiting system \eqref{eq:limiting-xy}.
        \item If $c_2<1$, then $(0,0)$ is a uniformly strong repeller for $\{(x,z):x>0,z\geq0\}$ for the limiting system \eqref{eq:limiting-xz}.
    \end{enumerate} 
\end{lemma}

\begin{proof}
    We only prove (a). For any $x_0>0$ and $y_0\geq0$, it is easy to see that
    \begin{align*}
        x'=x(1-x)-\frac{c_1mxy}{a_1x+my}\geq x((1-c_1)-x),
    \end{align*}
    which provides that $\liminf\limits_{t\to\infty}x(t)\geq 1-c_1>0$. Therefore, the proof is complete.
\end{proof}

\begin{lemma}\label{lemma:E1-saddle-limiting}~
    \begin{enumerate}
        \item If $\mcI_y^1>0$, then the equilibrium $(1,0)$ of the limiting system \eqref{eq:limiting-xy} is a saddle.
        \item If $\mcI_z^1>0$, then the equilibrium $(1,0)$ of the limiting system \eqref{eq:limiting-xz} is a saddle.
    \end{enumerate}
\end{lemma}

\begin{proof}
    We only prove (a). The Jacobian matrix at $(1,0)$ is
    \begin{align*}
        J(1,0)=\begin{bmatrix}
            -1 & -\frac{c_1m}{a_1}\\
            0 & \mcI_y^1
        \end{bmatrix},
    \end{align*}
    whose eigenvalues are $-1$ and $\mcI_y^1>0$.
\end{proof}

\begin{lemma}\label{lemma:limiting-no-periodic}~
    \begin{enumerate}
        \item If $c_1<1$, $a_1\leq e_1m$, and $0<\mcI_y^1<\frac{e_1m}{a_1}$, then there is no periodic solution and separatrix cycle for \eqref{eq:limiting-xy}.
        \item If $c_2<1$, $a_2\leq e_2$, and $0<\mcI_z^1<\frac{e_2}{a_2}$, then there is no periodic solution and separatrix cycle for \eqref{eq:limiting-xz}.
    \end{enumerate}
\end{lemma}

\begin{proof}
    For (a), we first show that there is no periodic solution. Take $\varphi(x,y)=\frac{1}{xy}$, then
    \begin{align*}
        \frac{\partial(\varphi x'(t))}{\partial x}+\frac{\partial(\varphi y'(t))}{\partial y}=-\frac{1}{y}+\frac{c_1m(a_1-e_1m)}{(a_1x+my)^2}<0.
    \end{align*}
    Hence, by Bendixon-Dulac criterion, there is no periodic solution. On the other hand, the Jacobian matrix at $(x_1,y_1)$ is
    \begin{align*}
        A_{xy}=\begin{bmatrix}
            1-2x_1-\frac{(a_1\mcI_y^1)^2}{e_1^2c_1m^2} & -\frac{a_1d_1^2}{e_1^2c_1m}\\
            \frac{(a_1\mcI_y^1)^2}{e_1c_1m^2} & -\frac{a_1d_1\mcI_y^1}{e_1c_1m}
        \end{bmatrix},
    \end{align*}
    whose trace and determinant are
    \begin{align*}
        \tr(A_{xy})&=
        a_1\mcI_y^1\left(\frac{e_1c_1m+a_1d_1-e_1d_1m}{e_1^2c_1m^2}\right)-1=-x_1+\frac{a_1d_1\mcI_y^1(a_1-e_1m)}{e_1^2c_1m^2}<0,\\
        \det(A_{xy}))&=\frac{a_1^2d_1\mcI_y^1}{e_1^2c_1m^2}\left(\frac{e_1m}{a_1}-\mcI_y^1\right)>0.
    \end{align*}
    Hence, $(x_1,y_1)$ is locally asymptotically stable. Moreover, since $(0,0)$ is a strong repeller for $\{(x,y):x>0,y\geq0\}$ (\Cref{lemma:E0-repeller-limiting}), and it is easy to see that $(0,0)$ is an attractor for $\{(x,y):x=0,y\geq0\}$, so there is no homoclinic loop of $(0,0)$. For $(1,0)$, since $\mcI_y^1>0$, there is an open neighborhood $U$ of $(1,0)$ such that if $(x(t),y(t))\in U$, then 
    ${y'(t)}/{y(t)}>{\mcI_y^1}/{2}>0$.
    Therefore, there is no $y_0>0$ such that the solution tends to $(1,0)$. In addition, for $x_0>0$ and $y_0=0$, it is easy to see that $(1,0)$ is globally asymptotically stable. Thus, there is no homoclinic loop of $(1,0)$.
    
    Finally, since $(0,0)$ is a uniformly strong repeller and $(x_1,y_1)$ is locally asymptotically stable, there is no heteroclinic cycle. Hence, there is no separatrix cycle on $\{(x,y):x\geq0,y\geq0\}$.
    
    For (b), we can use a similar argument as in (a) to show that there is no periodic solution. The Jacobian matrix of \eqref{eq:limiting-xz} at the positive equilibrium $(x_2,z_2)$ is
    \begin{align*}
        A_{xz}=\begin{bmatrix}
            1-2x_2-\frac{(a_2\mcI_z^1)^2}{e_2^2c_2} & -\frac{a_2d_2^2}{e_2^2c_2}\\
            \frac{(a_2\mcI_z^1)^2}{e_2c_2} & -\frac{a_2d_2\mcI_z^1}{e_2c_2}
        \end{bmatrix},
    \end{align*}
    whose trace and determinant can be expressed as
    \begin{align*}
        \tr(A_{xz})=-x_2+\frac{a_2d_2\mcI_z^1(a_2-e_2)}{e_2^2c_2}<0\quad\textup{and}\quad
        \det(A_{xz})=\frac{a_2^2d_2\mcI_z^1}{e_2^2c_2}\left(\frac{e_2}{a_2}-\mcI_z^1\right)>0.
    \end{align*}
    By a similar argument as in (a), \eqref{eq:limiting-xz} has no separatrix cycle.
\end{proof}

\subsection{Intrinsic Extinction of Consumer}\label{subsec:Extinction-consumer1}

First, we provide the extinction when the consumer can not sustain themselves. 
\begin{lemma}\label{lemma:consumer-cant-sustain}
    If $\mcI_y^1\leq0$, then $\lim\limits_{t\to\infty}y(t)=0$. On the other hand, if $\frac{e_2c_2}{a_2}+\mcI_z^y<0$, then $\lim\limits_{t\to\infty}z(t)=0$.
\end{lemma}
\begin{proof}
    From \eqref{eq:main-model-2}, we have
    \begin{align*}
        y'\leq y\left(\frac{e_1c_1mx}{a_1x+my+z}-d_1\right)\leq\frac{-d_1y(my+z)}{a_1x+my+z}\leq\frac{-d_1my^2}{a_1M_x+mM_y+M_z}=-\iota y^2,
    \end{align*}
    where $\iota=\frac{d_1m}{a_1M_x+mM_y+M_z}>0$. By using differential inequality, we obtain
    \begin{align*}
        y(t)\leq\frac{y_0}{1+\iota y_0t},\quad t>0.
    \end{align*}
    Hence, $\lim\limits_{t\to\infty}y(t)=0$. The second statement can be proved by a similar argument.
\end{proof}

The conditions $\mcI_y^1\leq0$ and $\frac{e_2c_2}{a_2}+\mcI_z^y<0$ represent the negative ability to sustain the survival of each consumer. The next lemma shows that, to make both consumers extinct, the condition $\frac{e_2c_2}{a_2}+\mcI_z^y<0$ in \Cref{lemma:consumer-cant-sustain} is not sharp enough. When the density of IG prey is low, then the IGP structure provides that IG predator has almost only one pathway to sustain itself. 


\begin{lemma}\label{lemma:consumer-extinction}
    If $\mcI_y^1\leq0$ and $\mcI_z^1<0$, then we have $\lim\limits_{t\to\infty}(y(t),z(t))=(0,0)$ for any $x_0,y_0,z_0\geq0$.
\end{lemma}
\begin{proof}
    From \Cref{lemma:consumer-cant-sustain}, we have $y(t)\to0$ as $t\to\infty$. If $z_0=0$, then $z(t)=0$ for all $t>0$, so we may assume $z_0>0$. Let $\delta_1=-\mcI_y^1\geq0$ and $\delta_2=-\mcI_z^1>0$. From \eqref{eq:main-model-3}, we have
    \begin{align*}
        z'=\frac{e_2c_2xz}{a_2x+my+z}+\frac{e_3c_3yz}{a_3y+z}-d_2z\leq-\delta_2z+e_3c_3y.
    \end{align*}
    Then by differential inequality, we obtain
    \begin{align*}
        z(t)\leq z_0e^{-\delta_2t}+e^{-\delta_2t}\int_0^te_3c_3y(s)e^{\delta_2s}\,ds.
    \end{align*}
    The second term will tend to zero by using L'Hopital's rule. Thus, $z(t)\to0$ as $t\to\infty$.  
    \end{proof}

\Cref{lemma:consumer-extinction} can not ensure whether the solution will tend to $E_0$ or $E_1$ (see \Cref{fig:bistability-of-extinction}). To guarantee the global stability of $E_1$, we need the following lemma, which can be proved by a similar argument as in \Cref{lemma:E0-repeller-limiting}.

\begin{lemma}\label{lemma:E0-repeller}
    If $\max\{c_1,c_2\}<1$, then $(0,0,0)$ is a uniformly strong repeller for $\{(x,y,z):x>0,y\geq0,z\geq0\}$.
\end{lemma}

\begin{proof}
    From \eqref{eq:main-model-1}, we have
   \begin{align*}
        \frac{x'}{x}
        \geq\left(1-\frac{\max\{c_1,c_2\}(my+z)}{\min\{a_1,a_2\}x+my+z}\right)-x\geq\left(1-\max\{c_1,c_2\}\right)-x.
    \end{align*}
    Hence, we have $\liminf\limits_{t\to\infty}x(t)\geq1-\max\{c_1,c_2\}>0$.
\end{proof}

The following corollary can be shown by using \Cref{lemma:consumer-extinction} and \Cref{lemma:E0-repeller}.

\begin{corollary}\label{cor:E1-stable}
    If $\max\{c_1,c_2\}<1$, $\mcI_y^1\leq0$, and $\mcI_z^1<0$, then for any $x_0>0$ and $y_0,z_0\geq0$, we have $\lim\limits_{t\to\infty}(x(t),y(t),z(t))=(1,0,0)$.
\end{corollary}

From the meaning of the two quantities $\mcI_y^1$ and $\mcI_z^1$, \Cref{lemma:consumer-extinction} and \Cref{cor:E1-stable} provide intuitive cases for the two predators' extinction. Next, we study consumer extinction when the two consumers are sustainable in their consumption of the resource.

\subsection{Exclusion-induced Extinction of Consumer}\label{subsec:Extinction-consumer2}

\begin{lemma}\label{lemma:y-extinct}
    If 
    $a_1\geq a_2$ and $\frac{a_2}{d_2}\mcI_z^1>\frac{a_1}{d_1}\mcI_y^1>0$, then for any $x_0,y_0\geq0$, $z_0>0$, we have $\lim\limits_{t\to\infty}y(t)=0$.
\end{lemma}

The proof of \Cref{lemma:y-extinct} is shown in \Cref{appendix :proofs}.

Under the conditions in \Cref{lemma:y-extinct}, the solution might tend to two equilibria: $E_0$ or $E_{xz}$ (see \Cref{fig:bistability-of-extinction}). We will provide a further discussion to determine whether the solution will tend to $E_0$ or $E_{xz}$ in \Cref{sec:competition-exclusion}.

Now, we provide another theorem for the extinction of the IG predator.

\begin{lemma}\label{lemma:z-extinct}
    If $a_2\geq a_1$ and
    \begin{align*}
        \frac{e_2c_2}{e_1c_1m}<\min\left\{\frac{e_3}{a_3},-\frac{\mcI_z^y}{d_1}\right\},
    \end{align*}
    then for any $x_0,z_0\geq0$ and $y_0>0$, we have $\lim\limits_{t\to\infty}z(t)=0$
\end{lemma}
\begin{proof}
    If $z_0=0$, then $z\equiv0$. For $z_0>0$, consider $V=\frac{z}{y^\eta}$, where $\eta=\frac{e_2c_2}{e_1c_1m}$, then 
    \begin{align*}
        \dot{V}(t)&=V(t)\left(\frac{e_2c_2x}{a_2x+my+z}-\frac{e_1c_1m\eta x}{a_1x+my+z}+\frac{e_3c_3-c_3a_3\eta}{a_3y+z}y-d_2+\eta c_3+\eta d_1\right)\\
        &\leq V(t)\left(\frac{e_2c_2x}{a_2x+my+z}-\frac{e_1c_1m\eta x}{a_1x+my+z}+\mcI_z^y-\eta c_3+\eta c_3+\eta d_1\right)\\
        &\leq V(t)\left(\mcI_z^y+\eta d_1\right).
    \end{align*}
    This provides that
        $z(t)\leq V(0)y(t)^\eta e^{(\mcI_z^y+\eta d_1)t}\leq V(0)M_y^\eta e^{(\mcI_z^y+\eta d_1)t}$.
    Therefore, $z(t)\to0$ as $t\to\infty$.
\end{proof}


\subsection{Total Extinction}\label{subsec:total-extinction}

If we use the Logistic function as the natural recruitment rate of the resource with Lotka-Volterra trophic functions in IGP models, no matter how poor the environment is, the model cannot support total extinction for almost all initial conditions \cite{hsu2015analysis}. In this section, we demonstrate that ratio-dependent trophic functions can lead to total extinction 
, we can conclude the following two lemmas.


\begin{lemma}\label{thm:total-extinction-1}
    Assume that $a_2\geq a_1$ and $\frac{e_2c_2}{e_1c_1m}<\min\left\{\frac{e_3}{a_3},-\frac{\mcI_z^y}{d_1}\right\}$. Moreover, if $c_1-1\geq\frac{a_1d_1}{e_1m}$, and $\frac{e_2c_2}{e_1c_1m}\geq1$,
    then for any $x_0,y_0,z_0>0$, we have $\lim\limits_{t\to\infty}(x(t),y(t),z(t))=(0,0,0)$.
\end{lemma}

\begin{proof}
    Consider $R(t)=\frac{x(t)}{y^\theta(t)}$, where $\theta=\frac{a_1}{e_1m}$. Then
    \begin{align*}
        R'(t)&=\frac{x}{y^\theta}\left(-(c_1-1-\theta d_1)-x+\frac{c_1z}{a_1x+my+z}-\frac{c_2z}{a_2x+my+z}+\frac{\theta c_3z}{a_3y+z}\right)\\
        &\leq\frac{x}{y^\theta}\left(-(c_1-1-\theta d_1)-x+\frac{z}{y}\left(\frac{c_1}{m}+\frac{\theta c_3}{a_3}\right)\right).
    \end{align*}
    Recall the proof of \Cref{lemma:z-extinct} and $\eta=\frac{e_2c_2}{e_1c_1m}\geq1$, $V(t)=\frac{z(t)}{y(t)^\eta}$, we have
    \begin{align*}
        R'(t)&\leq R(t)\left(-(c_1-1-\theta d_1)-x+\left(\frac{c_1}{m}+\frac{\theta c_3}{a_3}\right)V(0)M_y^{\eta-1}e^{(\mcI_z^y+\eta d_1)t}\right)\\
        &=:R(t)\left(-\delta-x+Me^{(\mcI_z^y+\eta d_1)t}\right),
    \end{align*}
    where $\delta=c_1-1-\theta d_1\geq0$ and $M=\left(\frac{c_1}{m}+\frac{\theta c_3}{a_3}\right)V(0)M_y^{\eta-1}$. Since $\mcI_z^y+\eta d_1<0$, by using the comparison, we obtain
    \begin{align*}
        R(t)\leq R(0)\exp\left(-\delta t-\int_0^tx(s)\,ds-\frac{2M}{\mcI_z^y+\eta d_1}\right)
    \end{align*}
    Now, if either (i) $\delta\neq0$ or (ii) $\delta=0$ and $\int_0^\infty x(s)\,ds=\infty$, then $x(t)\to0$ as $t\to\infty$. If $\delta=0$ and $\int_0^\infty x(s)\,ds<\infty$, then from \eqref{eq:main-model-1}, we have $x'\leq x$, and then
    \begin{align*}
        \int_{t-1}^tx(s)\,ds\geq x(t)\int_{t-1}^te^{s-t}\,ds=x(t)(1-e^{-1}),\quad t\geq1.
    \end{align*}
    Since $\int_0^\infty x(s)\,ds<\infty$, the left hand side will tend to $0$, that is, $x(t)\to0$ as $t\to\infty$. Furthermore, from the equation of $y$, we see that $y(t)\to0$ and thus $z(t)\to0$ as $t\to\infty$ as in the proof of Lemma~\ref{lemma:z-extinct}.
    Therefore, $(x(t),y(t),z(t))\to(0,0,0)$ as $t\to\infty$.
\end{proof}


By a similar argument and using  \Cref{lemma:y-extinct} instead of \Cref{lemma:z-extinct}, we can get the following result.

\begin{lemma}\label{thm:total-extinction-2}
    Assume that $a_1\geq a_2$ and $\frac{a_2}{d_2}\mcI_z^1>\frac{a_1}{d_1}\mcI_y^1>0$.
    If $c_2-1\geq\frac{a_2d_2}{e_2}$, and $a_2\mcI_z^1+a_1d_2\leq a_1\mcI_y^1+a_2d_1$,
    then we have $\lim\limits_{t\to\infty}(x(t),y(t),z(t))=(0,0,0)$ for any $x_0,y_0,z_0>0$.
\end{lemma}

\subsection{Conclusion of extinction}

From the discussions in Section~\ref{subsec:Extinction-consumer1}-- Section~\ref{subsec:total-extinction}, we obtain the following conclusion.

\begin{theorem}\label{thm:main-extinction}
    Let $\mcI_y^1$, $\mcI_z^1$, and $\mcI_z^y$ be defined in \Cref{tab:invasion-rate}, respectively. Then the following hold.
    \begin{enumerate}
        \item \textup{(Intrinsic extinction of one consumer, \Cref{lemma:consumer-cant-sustain})} If $\mcI_y^1\leq0$, then $\lim\limits_{t\to\infty}y(t)=0$. On the other hand, if $\frac{e_2c_2}{a_2}+\mcI_z^y<0$, then $\lim\limits_{t\to\infty}z(t)=0$.
        \item {\rm(Intrinsic extinction of two consumer, \Cref{lemma:consumer-extinction})} If $\mcI_y^1\leq0$ and $\mcI_z^1<0$, then for any $x_0,y_0,z_0\geq0$, we have $\lim\limits_{t\to\infty}(y(t),z(t))=(0,0)$.
        
        \item {\rm(Exclusion-induced extinction of one consumer, \Cref{lemma:y-extinct} and \Cref{lemma:z-extinct})} If the conditions in \Cref{lemma:y-extinct} hold, then for any $x_0,y_0\geq0$, $z_0>0$, $\lim\limits_{t\to\infty}y(t)=0$. On the other hand, if the conditions in \Cref{lemma:z-extinct} hold, then for any $x_0,z_0\geq0$, $y_0>0$, we have $\lim\limits_{t\to\infty}z(t)=0$.
        
        \item {\rm(Total extinction, \Cref{thm:total-extinction-2} and \Cref{thm:total-extinction-1})} If either the conditions in \Cref{thm:total-extinction-2} or \Cref{thm:total-extinction-1} hold, then for any $x_0,y_0,z_0>0$, we have $\lim\limits_{t\to\infty}(x(t),y(t),z(t))=(0,0,0)$.
    \end{enumerate}
\end{theorem}

\section{Competition Exclusion}\label{sec:competition-exclusion}

In addition to the extinction of both consumers, an important long-term outcome of the system is competitive exclusion, in which one consumer persists while the other is driven to extinction. In an intraguild predation system, such exclusion is determined not only by competition for the shared resource but also by the direct predation of the IG predator on the IG prey. In this section, we derive sufficient conditions such that one of the consumers goes extinct and the other one survives.

\subsection{Local Stability for Boundary Equilibria}\label{subsec:local-stability}

First, we study the local behavior at $E_{xy}$ and $E_{xz}$. The Jacobian matrix at $E_{xy}$ can be written as
\begin{align}\label{eq:Jacobian-xy}
    J(E_{xy})=\begin{bmatrix}
        A_{xy} & u_{xy}\\
        0 & \mcI_z^{xy}
    \end{bmatrix},
\end{align}
where $A_{xy}$ is mentioned in \Cref{lemma:limiting-no-periodic} and for some vector $u_{xy}$. Then we can obtain the following lemma.

\begin{lemma}\label{thm:Exy-las}
    Assume that $E_{xy}$ exists; that is,
    $0<\mcI_y^1<\frac{e_1m}{a_1}$. 
    If $\mathcal{I}_z^{xy}<0$ and  $a_1\leq e_1m$, then $E_{xy}$ is locally asymptotically stable. On the other hand, if $\mathcal I_z^{xy}>0$,
then $E_{xy}$ is unstable. 
\end{lemma}


On the other hand, the Jacobian matrix at $E_{xz}$ is
\begin{align}\label{eq:Jacobian-xz}
    J(E_{xz})=\begin{bmatrix}
        1-2x_2-\frac{(e_2c_2-a_2d_2)^2}{e_2^2c_2} & -\frac{c_1d_2m}{e_2c_2+(a_1-a_2)d_2}+\frac{d_2m(e_2c_2-a_2d_2)}{e_2^2c_2} & -\frac{a_2d_2^2}{e_2^2c_2}\\
        0 & \mcI_y^{xz} & 0\\
        \frac{(e_2c_2-a_2d_2)^2}{e_2c_2} & -\frac{md_2(e_2c_2-a_2d_2)}{e_2c_2}+e_3c_3 & \frac{a_2d_2^2}{e_2c_2}-d_2
    \end{bmatrix}.
\end{align}
We can conclude the following lemma by using \Cref{lemma:limiting-no-periodic} again.

\begin{lemma}\label{thm:Exz-las}
    Assume that $E_{xz}$ exists; that is, $0<\mcI_z^1<\frac{e_2}{a_2}$. If $\mathcal{I}_y^{xz}<0$ and $a_2\leq e_2$, then $E_{xz}$ is locally asymptotically stable. On the other hand, if $\mathcal{I}_y^{xz}>0$, then $E_{xz}$ is unstable. 
\end{lemma}

\Cref{thm:Exy-las} and \Cref{thm:Exz-las} answer when the IG predator and the IG prey can invade successfully. Nonetheless, the bistability might occur under the conditions of \Cref{thm:Exy-las} and \Cref{thm:Exz-las}. We provide the lemma in the next subsection to ensure the competition exclusion.

\subsection{Competition Exclusion}

In this subsection, we strengthen the results in \Cref{thm:Exy-las} and \Cref{thm:Exz-las} to the global version by using the theory of limiting system \cite{thieme1992convergence}.


\begin{lemma}\label{lemma:y-exclude}
    Assume that $E_{xy}$ exists and $\max\{c_1,c_2\}<1$. If either 
\begin{enumerate}
        \item $\frac{e_2c_2}{a_2}+\mcI_z^y<0$, $\mcI_y^1>c_3$, and $a_1\leq e_1m$, or
        \item $a_1\leq\min\{a_2,e_1m\}$, $c_3<\mcI_y^1<\frac{e_1m}{a_1}$, and $\frac{e_2c_2}{e_1c_1m}<\min\left\{\frac{e_3}{a_3},-\frac{\mcI_z^y}{d_1}\right\}$.
    \end{enumerate}
    Then $E_{xy}$ is globally asymptotically stable in $\{(x,y,z):x,y>0,z\geq0\}$.
\end{lemma}
\begin{proof}
    Since $z=0$ is invariant, we may assume the initial data $x_0,y_0,z_0>0$. For (a), \Cref{lemma:consumer-cant-sustain} provides that $z(t)\to0$ as $t\to\infty$. By \Cref{prop:well-posedness}, the forward solution is bounded. Moreover, since $z(t)\to0$, the $(x,y)$-components form a planar asymptotically autonomous system with limiting system \eqref{eq:limiting-xy}. Let $\omega$ be the $\omega$-limit set of the forward solution corresponding to $(x_0,y_0,z_0)$. Then \Cref{lemma:consumer-cant-sustain} provides that $\omega\subseteq\{(x,y,z):z=0\}$.

    By \cite[Theorem 1.6]{thieme1992convergence}, 
    $\omega$ is either
    \begin{enumerate*}
        \item[(i)] an equilibrium,
        \item[(ii)] a periodic orbit, or
        \item[(iii)] a union of equilibria and connecting orbits forming a cyclic structure.
    \end{enumerate*}
    From \Cref{lemma:limiting-no-periodic}, (ii) and (iii) are impossible. From \Cref{lemma:E0-repeller}, $E_0\not\in\omega$. Suppose that $E_1\in\omega$, then \eqref{eq:main-model-2} provides that
    \begin{align*}
        \frac{y'}{y}=\frac{e_1c_1mx}{a_1x+my+z}-\frac{c_3z}{a_3y+z}-d_1\geq\frac{e_1c_1mx}{a_1x+my+z}-c_3-d_1\to\mcI_y^1-c_3>0,
    \end{align*}
    as $t\to\infty$, which contradicts to $y\to0$ as $t\to\infty$. Therefore, $E_1\not\in\omega$. Finally, since
    \begin{align*}
        \mcI_z^{xy}=\frac{e_2c_2d_1}{e_1c_1m+(a_2-a_1)d_1}+\mcI_z^y\leq\frac{e_2c_2}{a_2}+\mcI_z^y<0,
    \end{align*}
    $E_{xy}$ is locally asymptotically stable. The Poincar\'{e}-Bendixson trichotomy \cite{thieme1992convergence} provides that $\omega=\{E_{xy}\}$, that is, every solution in $\{(x,y,z):x,y>0,z\geq0\}$ tend to $E_{xy}$.

    Since (b) can be proved by a similar argument, we omit the details. 
\end{proof}

From the proof of \Cref{lemma:y-exclude}, we obtained that the conditions are sufficient conditions for \Cref{thm:Exy-las}. Furthermore, \Cref{lemma:y-exclude} provides the condition such that the resource and the IG prey will survive, and the IG predator will exclude, which is the competition exclusion.

\begin{lemma}\label{lemma:z-exclude}
    Assume that $E_{xz}$ exists and $\max\{c_1,c_2\}<1$. If either
    \begin{enumerate}
        \item  $\mcI_y^1\leq0$, $\mcI_z^1>0$, and $a_2\leq e_2$, or
        \item  $\frac{a_2}{d_2}\mcI_z^1>\frac{a_1}{d_1}\mathcal I_y^1>0$ and $a_2\leq\min\{a_1,e_2\}$,
    \end{enumerate}
    then $E_{xz}$ is globally asymptotically stable in $\{(x,y,z):x,z>0,y\geq0\}$.
\end{lemma}

The above lemma can be proved by a similar argument in \Cref{lemma:y-exclude}.
Thanks to Lemma~\ref{lemma:y-exclude} and Lemma~\ref{lemma:z-exclude}, we obtain

\begin{theorem}[Competition-exclusion principle]\label{thm:competition-exclusion}
    Assume that $\max\{c_1,c_2\}<1$.
    $E_{xy}$ is globally asymptotically stable in $\{(x,y,z):x,y>0,z\geq0\}$ if 
    $E_{xy}$ exists and either (a) or (b) in Lemma~\ref{lemma:y-exclude} holds; $E_{xz}$ is globally asymptotically stable in $\{(x,y,z):x,z>0,y\geq0\}$ if $E_{xz}$ exists and either (a) or (b) in Lemma~\ref{lemma:z-exclude} holds.
\end{theorem}


\section{Positive Equilibria and Their Bifurcations}\label{sec:coexistence}

We begin by analyzing coexistence states by examining positive equilibria. Given a positive equilibrium $(x^*,y^*,z^*)$, from \eqref{eq:main-model-2} and \eqref{eq:main-model-3}, 
we have
\begin{align*}
    &\frac{e_1c_1mx^*}{a_1x^*+my^*+z^*}=\frac{c_3z^*}{a_3y^*+z^*}+d_1>d_1\quad\textup{and}
\quad \frac{e_2c_2x^*}{a_2x^*+my^*+z^*}=-\frac{e_3c_3y^*}{a_3y^*+z^*}+d_2<d_2,
\end{align*}
which gives
\begin{proposition}
Suppose that \eqref{eq:main-model-1}--\eqref{eq:main-model-3} admits a positive equilibrium, denoted by $(x^*,y^*,z^*)$. Then
\begin{align*}
    \frac{a_1}{d_1}\mcI_y^1>\frac{my^*+z^*}{x^*}>\frac{a_2}{d_2}\mcI_z^1,
\end{align*}
where $\frac{my^*+z^*}{x^*}$ is the total effective consumer-to-resource ratio. 
\end{proposition}


This inequality means that the IG prey has a resource-based advantage. The IG predator compensates for its resource-based disadvantage by feeding on the IG prey. This interpretation is qualitatively consistent with the classical coexistence criterion of Holt and Polis \cite{holt1997theoretical}.




It is difficult to compute the positive equilibrium by solving the algebraic equation directly; the stability of the positive equilibrium is also complicated to study. 
Instead of solving the exact form of positive equilibria, we combine the results in \Cref{subsec:local-stability} with bifurcation theory to investigate the local emergence of positive equilibria from boundary equilibria.
We then introduce a change of coordinates to 
obtain an explicit parameterization of the positive equilibrium branch and classify its
global structure in
$8$ generic parameter regimes.

At the end of this section, we study the stability corresponding to the relative predation efficiency $m$ and the invasion rate $\mcI_z^y$.

\subsection{Transcritical Bifurcation}

In this subsection, we study the transcritical bifurcation from $E_{xz}$ as a function of the parameter $m$, since $E_{xz}$ is independent of $m$.

To avoid the degeneration of the bifurcation, we consider the following generic assumptions
\begin{align}
    &\Delta_{xz}:=d_2^2(a_2-a_1)+ e_2c_2\left(\mcI_z^y+\frac{e_3d_1}{a_3}\right)\neq0,\tag{A1}\label{eq:generic-assumption-Delta-xz}\\
    &\Delta_{xy}:=(a_2-a_1)(\mcI_z^y)^2+e_2c_2\left( \mcI_z^y+\frac{e_3d_1}{a_3}\right)\neq0,\quad\textup{and}\quad\mcI_z^y\neq0\tag{A2}\label{eq:generic-assumption-I_z^y}
\end{align}

From \Cref{subsec:local-stability}, we have $E_{xz}$
is locally asymptotically stable if $0<\mcI_z^1<\frac{e_2}{a_2}$, $a_2\leq e_2$, and $\mcI_y^{xz}<0$. It is easy to solve that
\begin{align}\label{m-xz}
\mathcal{I}_y^{xz}=0\quad \text{ if and only if }  \quad  m=m_{xz}:=\frac{c_3+d_1}{e_1c_1d_2}(e_2c_2+(a_1-a_2)d_2).
\end{align}
From this value, we can get a local transcritical bifurcation at $E_{xz}$ by applying Sotomayor's theorem \cite{sotomayor1973generic} on the equilibrium $E_{xz}$.

\begin{theorem}\label{thm:bifurcation-Exz}
    Assume that $0<\mcI_z^1<\frac{e_2}{a_2}$ and $a_2\leq e_2$. If \eqref{eq:generic-assumption-Delta-xz} holds, then the system \eqref{eq:main-model-1}--\eqref{eq:main-model-3} undergoes a local transcritical bifurcation at $m=m_{xz}$. Moreover, when $m$ is near $m_{xz}$, we have
    \begin{enumerate}
        \item if $\Delta_{xz}>0$, then there is a locally asymptotically stable positive equilibrium branch on $m>m_{xz}$;
        \item if $\Delta_{xz}<0$, then there is an unstable positive equilibrium branch on $m<m_{xz}$.
    \end{enumerate}
\end{theorem}
\begin{proof}
Let $\beta:=e_2c_2-a_2d_2$. Since $E_{xz}$ exists, we have $0<\beta<e_2$. Let us recall $J(E_{xz})$ from \eqref{eq:Jacobian-xz} and write 
$J(E_{xz})=J(E_{xz};m)$ to emphasize the dependence on $m$. Note that when $m=m_{xz}$, we have $\mathcal{I}_y^{xz}(m_{xz})=0$, which forms an zero eigenvalue of $J(E_{xz})=J(E_{xz};m_{xz})$.

Let $u$ be the eigenvector of $J(E_{xz};m_{xz})$ and $v$ be the eigenvector of $J(E_{xz};m_{xz})^\top$ corresponding to the zero eigenvalue $\mathcal{I}_y^{xz}(m_{xz})=0$. Note that under the assumption of Theorem~\ref{thm:bifurcation-Exz}, the zero eigenvalue is simple.
Let us choose $u$ and $v$ as 
$u=\begin{bmatrix}
        u_1 & 1 & u_3
    \end{bmatrix}^\top,\quad v=\begin{bmatrix}
        0 & 1 & 0
    \end{bmatrix}^\top$.

    The normal form \cite{sotomayor1973generic,buonomo2015note} of \eqref{eq:main-model-1}--\eqref{eq:main-model-3} can be written as
    \begin{align}\label{eq:normal-form-transcritical}
        \dot{\xi}=a\xi^2+b(m-m_{xz})\xi+\xi O(\|(\xi,m-m_{xz})\|^2),
    \end{align}
    where
    \begin{align*}
        &a=\frac12v^\top D^2\bm{F}(E_{xz},m_{xz})(u,u)=\frac{e_2c_3a_3\Delta_{xz}}{\bigl(e_2c_2+(a_1-a_2)d_2\bigr)\beta(\beta-e_2)}\neq0,\\
        &b=v^\top D\bm{F}_m(E_{xz},m_{xz})u=\frac{e_1c_1d_2}{e_2c_2+(a_1-a_2)d_2}>0.
    \end{align*}
    Then it is easy to see that 
    \begin{enumerate}
        \item If $\Delta_{xz}>0$, then there is a locally asymptotically stable positive equilibrium when $m$ is near $m_{xz}$ and $m>m_{xz}$.
        \item If $\Delta_{xz}<0$, then there is a unstable positive equilibrium when $m$ is near $m_{xz}$ and $m<m_{xz}$.
    \end{enumerate}

    This completes the proof.
\end{proof}


Similarly, the threshold for the invasion of the IG predator $z$ into 
$E_{xy}$ is obtained by solving $\mathcal I_z^{xy}=0$.  A direct computation gives
\begin{align}\label{m-xy}
\mathcal I_z^{xy}=0 \quad \text{ if and only if }\quad   m=m_{xy} := \frac{d_1}{e_1c_1} \left[ (a_1-a_2)-\frac{e_2c_2}{\mcI_z^y} \right].  
\end{align}

Different from $E_{xz}$, the equilibrium $E_{xy}$ itself depends on the bifurcation parameter $m$. So, a direct application of the standard local bifurcation calculation in the original coordinates is less convenient. Instead, in the next subsection, we follow the positive equilibrium branch bifurcating from $E_{xz}$ and determine whether this branch connects to $E_{xy}$ or can be continued for arbitrarily large values of $m$.





We first exclude accumulation at $E_0$ and $E_1$ under certain conditions.
Denote the set of positive equilibria by
\begin{align*}
    \Gamma=\{(E^*,m)\in(0,\infty)^4:f(E^*,m)=0\}.
\end{align*}

\begin{lemma}
    If $\max\{c_1,c_2\}<1$ and $\mcI_z^1>0$, then 
    there is no sequence $(E_n^*,m_n)\in\Gamma$ satisfying $E_n^*\to E_0$ or $E_n^*\to E_1$ as $n\to\infty.$ 
\end{lemma}

\begin{proof}
    From \Cref{lemma:E0-repeller}, $E_n^*$ will not converge to $E_0$. From \eqref{eq:main-model-3}, we have $\frac{my+z}{x}\geq\frac{e_2c_2-a_2d_2}{d_2}$. The following inequality completes the proof:
    \begin{align*}
        1-x\geq\frac{\min\{c_1,c_2\}\cdot\frac{my+z}{x}}{\max\{a_1,a_2\}+\frac{my+z}{x}}\geq\frac{\min\{c_1,c_2\}\cdot\frac{e_2c_2-a_2d_2}{d_2}}{\max\{a_1,a_2\}+\frac{e_2c_2-a_2d_2}{d_2}}>0.
    \end{align*}
\end{proof}



\subsection{Global Structure of the Positive Equilibrium Branch}\label{subsec:global-structure-positive-equilibrium}

In this subsection, we 
study the global structure of $\Gamma$ by deriving an explicit parameterization of $\Gamma$. In particular, we determine whether this branch emanating from $E_{xz}$ connects to $E_{xy}$ or is continued to $m=\infty$. We also discuss the possible fold points along the branch (in this article, ``fold'' means the point with zero first-order derivative and nonzero second-order derivative along the curve).

For any fixed $m>0$, we introduce the change of variables: 
\begin{align}\label{eq:change-of-coordinate}
    (x,v,w):=\Big(x,\frac{my}{my+z},\frac{my+z}{x}\Big)=\Phi_m(x,y,z).
\end{align}
Here, $v$ is the fraction of the total weighted consumer abundance contributed by the IG prey;
while $w$ is the weighted consumer-to-resource ratio.

An aggregate prey-to-predator ratio was previously given in \cite{hsu2001rich}, which introduced $u:=\frac{x}{y+z}$ as a new state variable ($m=1$ therein), replacing the resource abundance $x$ while retaining the two predator variables $y$ and $z$ and facilitating the qualitative analysis of the resulting system consisting of $(u,y,z)$. 
Different from \cite{hsu2001rich}, we retain $x$ and replace $y$ and $z$ by new variables. 
Inspired by \cite{hsu2001rich}, the corresponding weighted aggregate quantity in our case is $my+z$, which suggests introducing its ratio to the resource $w=\frac{my+z}{x}$. 
The variables $x$ and $w$, however, determine only the total weighted consumer abundance $my+z=wx$
and cannot distinguish how this quantity is divided between the IG
prey $y$ and the IG predator $z$. To retain the composition information lost in this aggregation, we therefore introduce
$v=\frac{my}{my+z}$. The resulting coordinates $(x,v,w)$ separate resource abundance, consumer composition, and total weighted consumer abundance relative to the resource.  More importantly, they reduce the
positive-equilibrium equations to a one-dimensional curve parametrized
by $w$, which enables the global branch analysis below.

It is easy to check that $x,w>0$, $v\in(0,1)$ iff $(x,y,z)\in\mathbb{R}^3_{>0}$. So $\Phi_m$ maps $\mathbb R^3_{>0}$ onto 
\[
\Omega:=\{(x,v,w)|\, x>0,\ 0<v<1,\ w>0\}.
\]
We can also find its inverse, given by 
\[ 
\Phi_m^{-1}(x,v,w)=\Big( x,\frac{vwx}{m},(1-v)wx \Big). 
\] 
Therefore, the map $(x,y,z,m)\mapsto(\Phi_m(x,y,z),m)$ is a smooth diffeomorphism from $\mathbb R^3_{>0}\times \mathbb R_{>0}$
onto $\Omega\times \mathbb R_{>0}$. Hence, it preserves the local bifurcation type.
By direct computations, 
the ODEs \eqref{eq:main-model-1}--\eqref{eq:main-model-3} is transformed into 
\begin{equation}\label{eq:transformed-system}
\begin{cases}
\dps   x'=xA=:\Tilde{X}(x,v,w)\\
\dps   v'=v(1-v)(B-C)=:\Tilde{V}(v,w)\\
\dps   w'=w(vB+(1-v)C-A)=:\Tilde{W}(x,v,w)
\end{cases}
\end{equation}
where
\begin{align*}
    &A(x,v,w)=1-x-\frac{c_1vw}{a_1+w}-\frac{c_2w(1-v)}{a_2+w},\quad B(v,w)=\frac{e_1c_1m}{a_1+w}-f_Y(v)-d_1,\quad \\
    &C(v,w)=\frac{e_2c_2}{a_2+w}+f_Z(v)-d_2,\quad\bm{\Tilde{F}}(x,v,w)=(\Tilde{X}(x,v,w),\Tilde{V}(v,w),\Tilde{W}(x,v,w)),\\
    &f_Y(v)=\frac{c_3m(1-v)}{m+(a_3-m)v},\quad\textup{and}\quad f_Z(v)=\frac{e_3c_3v}{m+(a_3-m)v}.
\end{align*}
It is easy to see that the following are equivalent
\begin{enumerate}
    \item $(x,y,z)\in\mathbb R^3_{>0}$ is a positive equilibrium of \eqref{eq:main-model-1}--\eqref{eq:main-model-3}
    \item $(x,v,w)\in\Omega$ is a positive equilibrium of \eqref{eq:transformed-system}
    \item $(x,v,w)\in\Omega$ satisfies $A(x,v,w)=B(v,w)=C(v,w)=0$.
\end{enumerate}

From $B=0$ and $C=0$, the relation between the positive equilibrium and the relative predation efficiency $m$ can be expressed as 
\begin{align}\label{eq:relation-m-and-w}
    m=m(w)=\frac{a_1+w}{e_1e_3c_1}\left(\frac{e_2c_2a_3}{a_2+w}+a_3\mcI_z^y+e_3d_1\right).
\end{align}

Let 
\[
w_{xy}=-\frac{e_2c_2}{\mcI_z^y}-a_2=\frac{m_{xy}y_1}{x_1},\quad w_{xz}=\frac{e_2c_2-a_2d_2}{d_2}=\frac{z_2}{x_2}>0.
\]
From $C=0$, we have
\begin{align*}
    d_2-\frac{e_2c_2}{a_2+w}=f_Z(v)>0,
\end{align*}
which force that $w$ is admissible only when $w>w_{xz}$.

It is also easy to see that $m(w_{xy})=m_{xy}$, $m(w_{xz})=m_{xz}$ and some basic properties for $\Gamma$:
\begin{align*}
    &m'(w)=\frac{(a_3\mcI_z^y+e_3d_1)(a_2+w)^2+e_2c_2a_3(a_2-a_1)}{e_1e_3c_1(a_2+w)^2},\quad m''(w)=\frac{2e_2c_2a_3(a_1-a_2)}{e_1e_3c_1(a_2+w)^3},\\
    &m'(w_{xy})=\frac{a_3\Delta_{xy}}{e_1e_2e_3c_1c_2},\quad\textup{and}\quad m'(w_{xz})=\frac{a_3\Delta_{xz}}{e_1e_2e_3c_1c_2}.
\end{align*}

 To find the fold on $\Gamma$, 
we assume that $a_3\mcI_z^y+e_3d_1\neq0$.
we can get 
\[
m'(w)=0\quad \text{ if and only if  } \quad (a_3\mcI_z^y+e_3d_1)(a_2+w)^2=e_2c_2a_3(a_1-a_2).
\]
Therefore, $m'(w)=0$ has a zero for $w>0$ if and only if 
\begin{align}\label{m'=0cond}
\frac{e_2c_2a_3(a_1-a_2)}{(a_3\mcI_z^y+e_3d_1)}>a_2^2.  \end{align}
Whenever \eqref{m'=0cond} holds, the zero of $m'$ is unique and satisfies 

\begin{equation}\label{eq:relation-between-wSN-mSN}
\begin{aligned}
    w=w_{SN}=\sqrt{\frac{e_2c_2a_3(a_1-a_2)}{a_3 \mcI_z^y+e_3d_1}}-a_2,\quad 
    m_{SN}=m(w_{SN})=\frac{e_2c_2a_3}{e_1e_3c_1}\frac{(a_1+w_{SN})^2}{(a_2+w_{SN})^2}.
\end{aligned}
\end{equation}
By \eqref{eq:relation-between-wSN-mSN},
we have $a_1\neq a_2$ and thus $m''(w_{SN})\neq 0$
In this case, $\Gamma$ has at most one fold in $\mathbb{R}_{>0}^4$. 
The admissibility of $w_{SN}$ will be characterized in
\Cref{thm:structure-branch}.

\begin{remark}~
    \begin{enumerate}
        \item If $a_1=a_2$, then 
        \eqref{eq:generic-assumption-Delta-xz} implies 
        $a_3\mcI_z^y+e_3d_1\neq 0$, and thus $m'(w)\neq 0$ for all $w>0$. Hence, no fold occurs.
        \item If $a_3\mcI_z^y+e_3d_1=0$, then from \eqref{eq:generic-assumption-Delta-xz}, we have $a_1\neq a_2$, so $m'(w)\neq0$ for all $w>0$. Hence, there is no fold on $\Gamma$.
        \item To ensure the existence of $w_{SN}$ in $\mathbb{R}$, we provide the assumption:
        \begin{align}\tag{A3}\label{eq:nondegeneracy-saddle-node}
            (a_1-a_2)(a_3\mcI_z^y+e_3d_1)>0.
        \end{align}
    \end{enumerate}
\end{remark}

\begin{theorem}\label{thm:structure-branch}
Assume that \eqref{eq:generic-assumption-Delta-xz} and \eqref{eq:generic-assumption-I_z^y} hold, $\max\{c_1,c_2\}<1$, $a_2\leq e_2$, $0<\mcI_z^1<\frac{e_2}{a_2}$, and the definition of $w_{xy}$, $w_{xz}$ are the same as above.
\begin{enumerate}
    \item Then $\Gamma$ can be parametrized by $w$ on the interval
    \begin{align*}
    I:=\begin{cases}
        (w_{xz},w_{xy}),&\text{if } \mcI_z^y<0;\\
        (w_{xz},\infty),&\text{if } \mcI_z^y>0.
    \end{cases}
    \end{align*}
    Furthermore, if $\mcI_z^y<0$, then $E^*(w_{xz}^+)=E_{xz}$ and $E^*(w_{xy}^-)=E_{xy}$.
    \item $\Gamma$ has a unique fold at $w=w_{SN}$ if and only if one of the following holds:
    \begin{enumerate}[label=(\arabic*)]
        \item $\Delta_{xy}\Delta_{xz}<0$ and $\mcI_z^y<0$
        \item $\Delta_{xz}<0$ and $\mcI_z^y>0$
    \end{enumerate}
\end{enumerate}
\end{theorem}

\begin{proof}
    For (a), we have shown that $w>w_{xz}$ from the previous argument. From $B=C=0$, we have
    \begin{align*}
        1=\frac{1}{c_3}f_Y(v)+\frac{a_3}{e_3c_3}f_Z(v)=\frac{1}{c_3}\left(\frac{e_1c_1m}{a_1+w}-d_1\right)+\frac{a_3}{e_3c_3}\left(d_2-\frac{e_2c_2}{a_2+w}\right).
    \end{align*}
    Then
    \begin{align}\label{eq:proof-structure}
        \frac{a_3}{e_3}\left(\mcI_z^y+\frac{e_2c_2}{a_2+w}\right)=c_3-\frac{a_3d_2}{e_3}+\frac{a_3}{e_3}\frac{e_2c_2}{a_2+w}=\frac{e_1c_1m}{a_1+w}-d_1>0.
    \end{align}
    Therefore, when $\mcI_z^y<0$, then \eqref{eq:proof-structure} holds if $w<w_{xy}$. When $\mcI_z^y>0$, then \eqref{eq:proof-structure} holds for all $w>w_{xz}$. It is also easy to see that for any $w\in I$, there is a unique positive equilibrium $E^*=(x,v,w)\in\Omega$.
    Furthermore, when $w\to w_{xz}^+$, we have
    \begin{align*}
        d_2-\frac{e_2c_2}{a_2+w}\to0,
    \end{align*}
    so that $v\to0$ and $(x(w),y(w),z(w))\to(x_2,0,z_2)=E_{xz}$. Hence, $E^*(w_{xz}^+)=E_{xz}$. By a similar argument, we have $E^*(w_{xy}^-)=E_{xy}$. 
    
    For (b), the uniqueness is easy to see from the expression of $m'(w)$. When \textit{(1)} holds, $\Delta_{xy}\Delta_{xz}<0$ implies $m'(w_{xy})m'(w_{xz})<0$, and $\mcI_z^y<0$ implies $\Gamma$ connect $E_{xy}$ and $E_{xz}$. By intermediate value theorem, there is $\overline{w}\in(w_{xz},w_{xy})$ such that $m'(\overline{w})=0$. Since the root of $m'(w)$ is at most one in $\mathbb{R}_{>0}$, we have $\overline{w}=w_{SN}$. Moreover, since $m'(w_{xy})m'(w_{xz})<0$, we have $a_1\neq a_2$, so that $m''(w_{SN})\neq0$. Hence, $w=w_{SN}$ is a fold. When (2) holds, we have
    \begin{align*}
        \lim\limits_{w\to\infty}m'(w)=\frac{a_3\mcI_z^y+e_3d_1}{e_1e_3c_1}>0.
    \end{align*}
    Since $\Delta_{xz}<0$, we have $a_1>a_2$. By a similar argument as above, there is a fold at $w=w_{SN}$. In other cases, the expression of $w_{SN}$ provides that there is no fold on $\Gamma$.
\end{proof}

From \Cref{thm:structure-branch}, the classification of the topological structure can be listed in \Cref{tab:classification-bifurcation-and-coexistence}. The corresponding bifurcation diagrams can be found in \Cref{fig:forward} and \Cref{fig:backward}. However, the stability indicated in \Cref{fig:forward} and \Cref{fig:backward} is not guaranteed yet. The stability analysis will be provided in the next subsection.

\begin{table}
    \centering
        \begin{tabular}{|c|c|c|c|c||c|c|c|c|c|}\hline
            \multicolumn{5}{|c||}{$\mcI_z^y>0$, interval: $(w_{xz},\infty)$} & \multicolumn{5}{c|}{$\mcI_z^y<0$, interval: $(w_{xz},w_{xy})$}\\\hline
            case & $\Delta_{xy}$ & $\Delta_{xz}$ & fold & figure & case & $\Delta_{xy}$ & $\Delta_{xz}$ & fold  & figure\\\hline\hline
            (i) & $>0$ & $>0$ & no & \Cref{fig:forward-1} & (v) & $>0$ & $>0$ & no & \Cref{fig:forward-2} \\\hline
            (ii) & $>0$ & $<0$ & yes & \Cref{fig:backward-3} & (vi) & $>0$ & $<0$ & yes & \Cref{fig:backward-2}\\\hline
            (iii) & $<0$ & $>0$ & no & \Cref{fig:forward-1} & (vii) & $<0$ & $>0$ & yes & \Cref{fig:forward-3}\\\hline
            (iv) & $<0$ & $<0$ & yes & \Cref{fig:backward-3} & (viii) & $<0$ & $<0$ & no & \Cref{fig:backward-1}\\\hline
        \end{tabular}
    \caption{The words ``yes'' and ``no'' in the fold column mean $\Gamma$ has and does not have the fold in $I$, respectively.}
    \label{tab:classification-bifurcation-and-coexistence}
\end{table}

\begin{figure}
    \centering
    \begin{subfigure}[t]{.3\textwidth}
        \includegraphics[width=\linewidth]{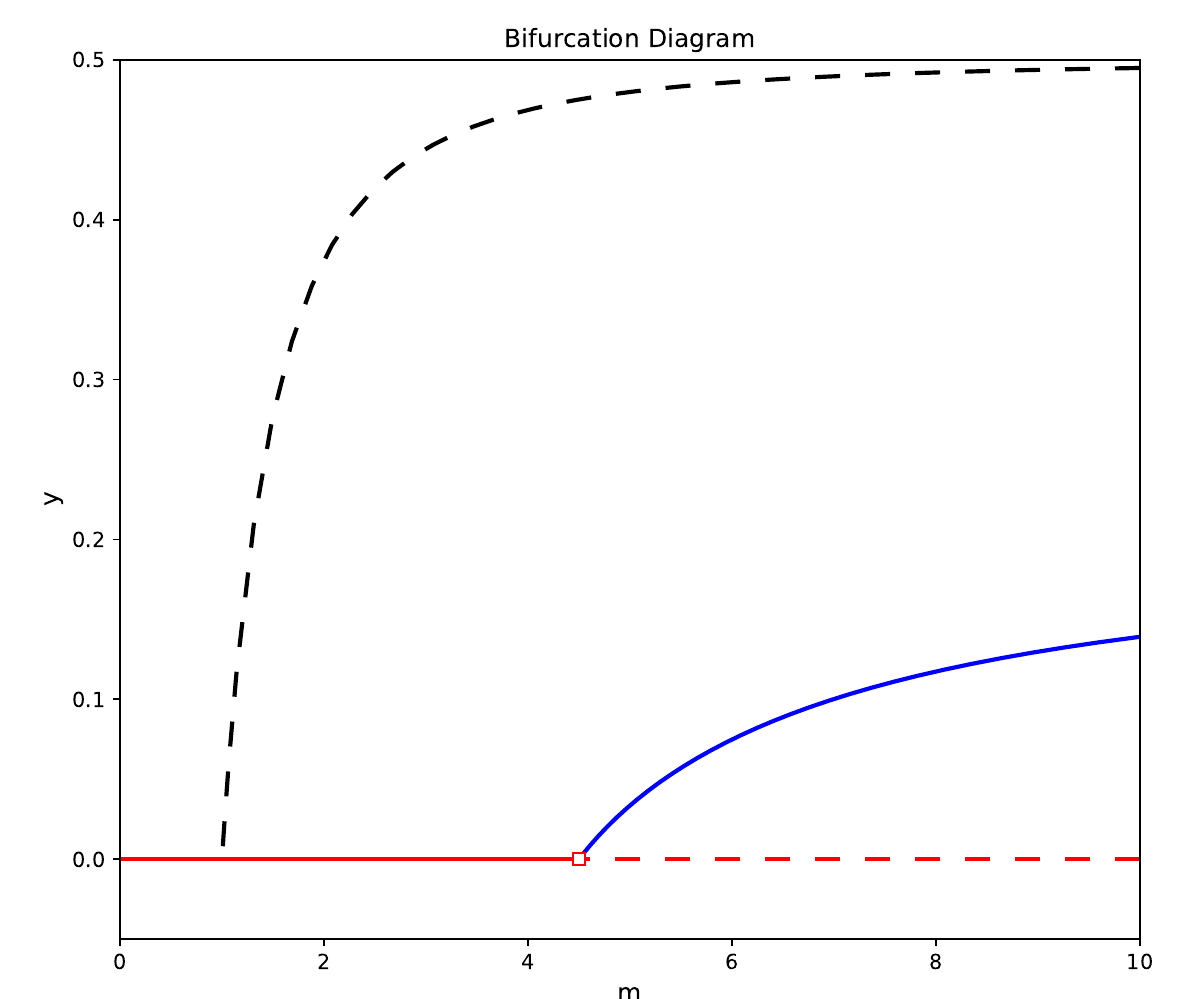}
        \caption{Cases (i) and (iii)}
        \label{fig:forward-1}
    \end{subfigure}\hspace{.3cm}
    \begin{subfigure}[t]{.3\textwidth}
        \includegraphics[width=\linewidth]{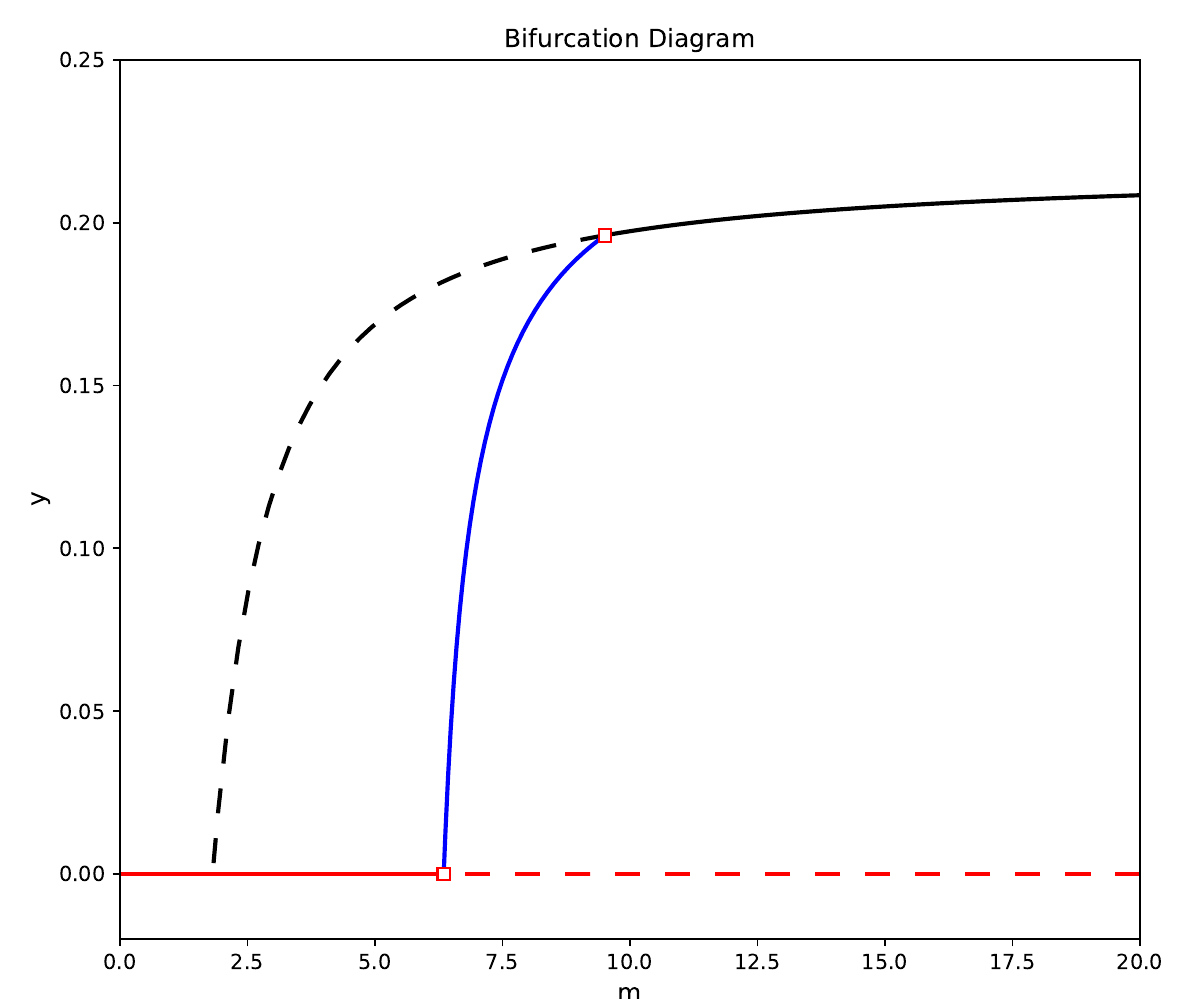}
        \caption{Case (v)}
        \label{fig:forward-2}
    \end{subfigure}\hspace{.3cm}
    \begin{subfigure}[t]{.3\textwidth}
        \includegraphics[width=\linewidth]{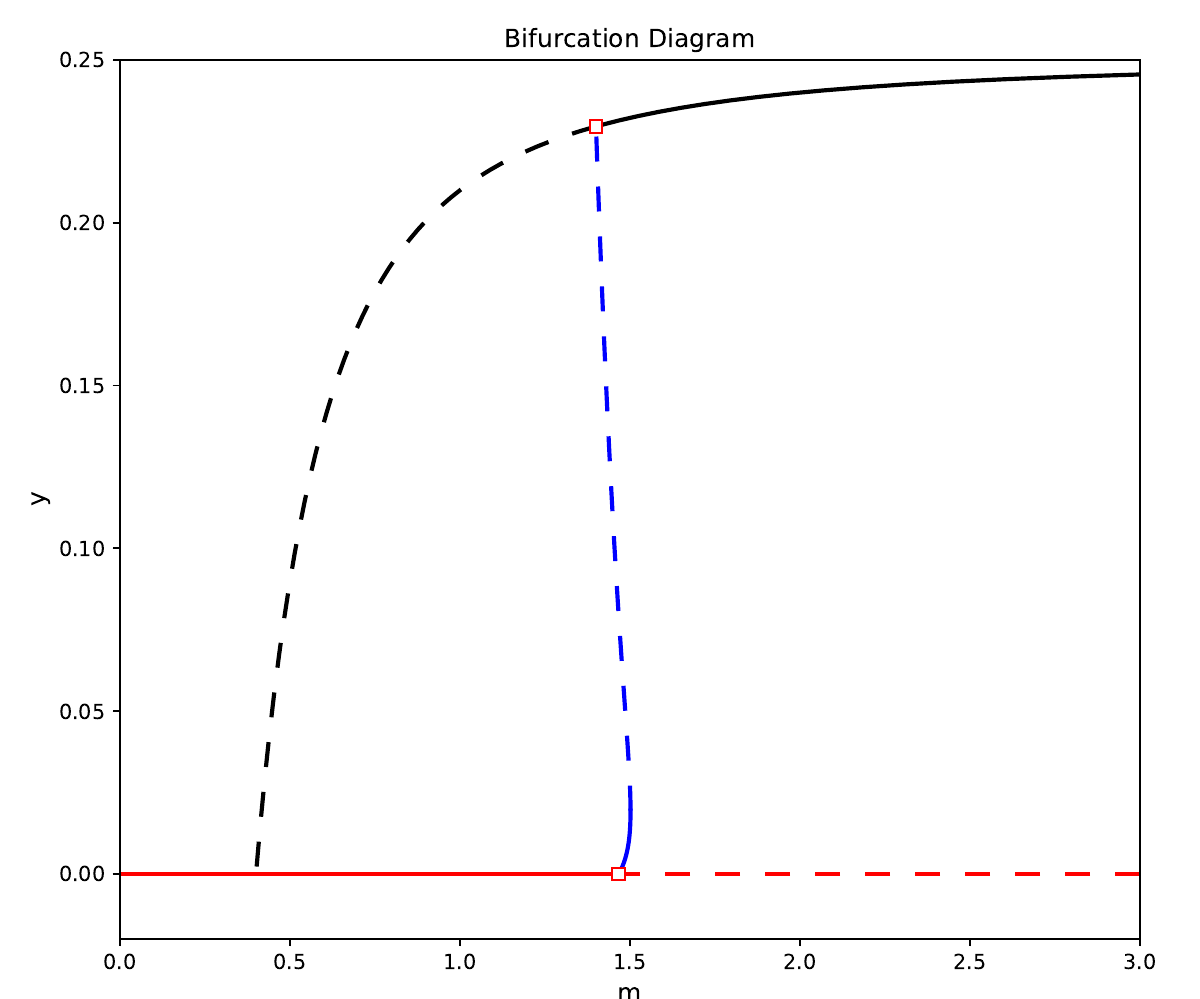}
        \caption{Case (vii)}
        \label{fig:forward-3}
    \end{subfigure}
    \hfill
    \caption{Bifurcation diagrams for the forward bifurcation cases (only the topological structure is ensured). The red line, black line, and blue line in the above figure represent the boundary equilibrium $E_{xz}$, $E_{xy}$, and positive equilibrium $E^*$, respectively.}
    \label{fig:forward}
\end{figure}

\begin{figure}
    \centering
    \begin{subfigure}[t]{.3\textwidth}
        \includegraphics[width=\linewidth]{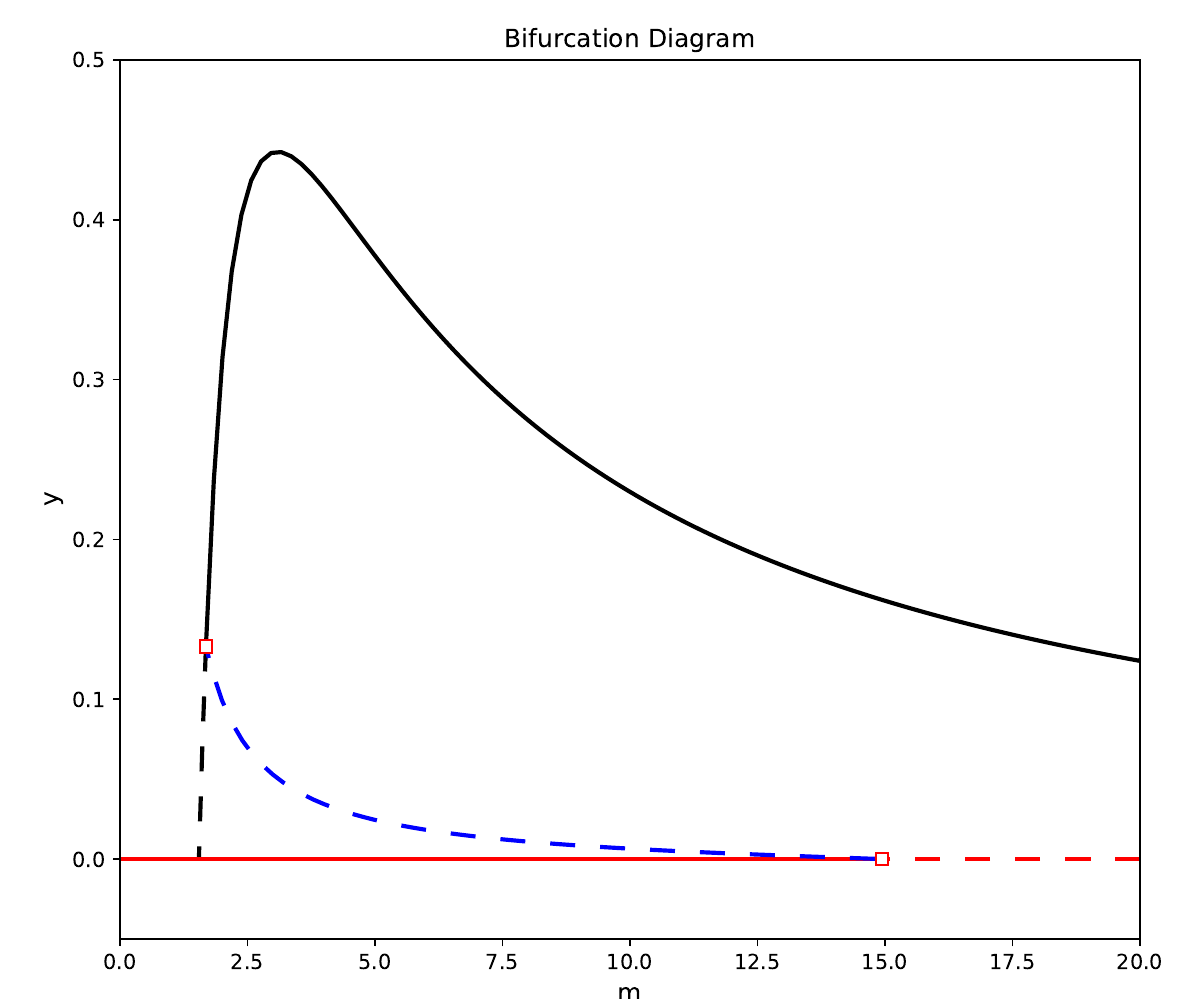}
        \caption{Case (viii)}
        \label{fig:backward-1}
    \end{subfigure}\hspace{.3cm}
    \begin{subfigure}[t]{.3\textwidth}
        \includegraphics[width=\linewidth]{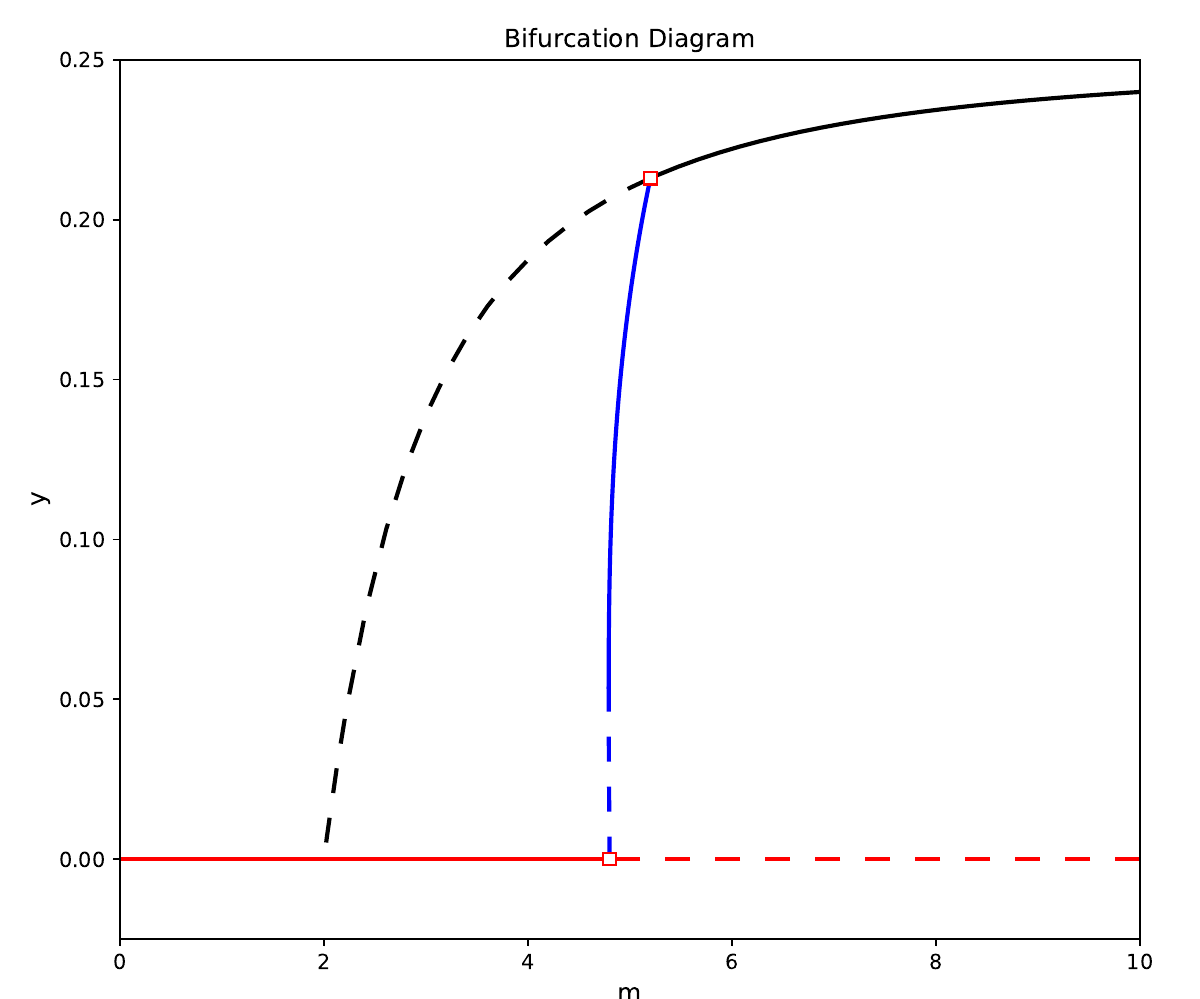}
        \caption{Case (vi)}
        \label{fig:backward-2}
    \end{subfigure}\hspace{.3cm}
    \begin{subfigure}[t]{.3\textwidth}
        \includegraphics[width=\linewidth]{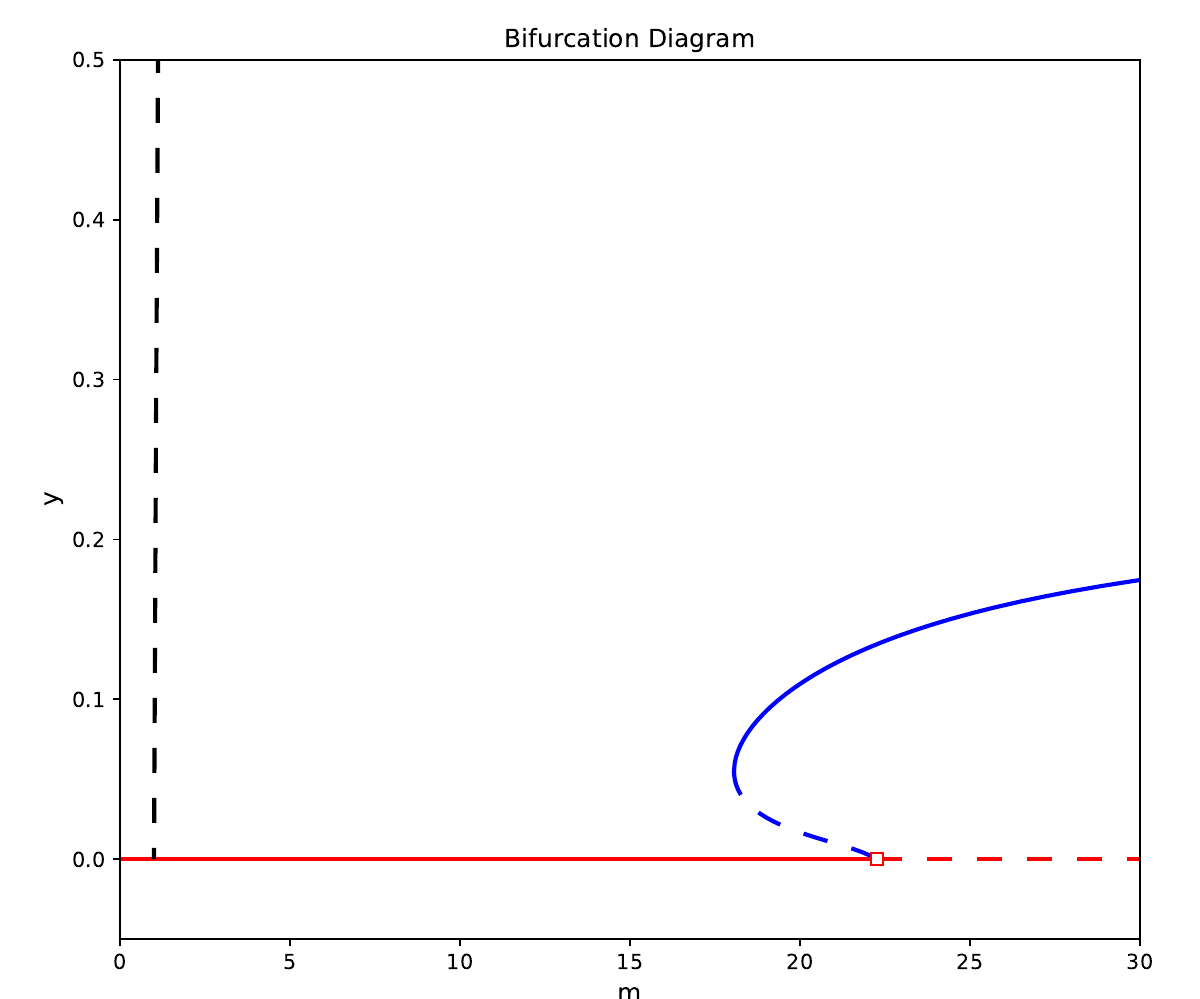}
        \caption{Cases (ii) and (iv)}
        \label{fig:backward-3}
    \end{subfigure}
    \hfill
    \caption{Bifurcation diagrams for the backward bifurcation cases (only the topological structure is ensured). The red line, black line, and blue line in the above figure represent the boundary equilibrium $E_{xz}$, $E_{xy}$, and positive equilibrium $E^*$, respectively.}
    \label{fig:backward}
\end{figure}

\subsection{Local Stability of the Positive Equilibrium}\label{subsec:local-stability-positive-equilibrium}

To study the local stability of the positive equilibria, the characteristic equation of the Jacobian matrix at positive equilibrium $(x,v,w)$ of \eqref{eq:transformed-system} is
\begin{align}\label{eq:characteristic-equation}
    \lambda^3+\alpha_1(w)\lambda^2+\alpha_2(w)\lambda+\alpha_3(w)=0,
\end{align}
where
\begin{align*}
    \alpha_1(w)&=x+\frac{c_3m(e_3-a_3)v(1-v)}{(m+(a_3-m)v)^2}+w\left(\frac{c_1(e_1m-a_1)v}{(a_1+w)^2}+\frac{c_2(e_2-a_2)(1-v)}{(a_2+w)^2}\right),\\
    \alpha_2(w)&=xv(1-v)(f_Y'(v)+f_Z'(v))\\
    &+xw\left(\frac{e_1c_1mv}{(a_1+w)^2}+\frac{e_2c_2(1-v)}{(a_2+w)^2}\right)+\frac{e_1e_3c_1c_3mm'v(1-v)w}{(a_1+w)(m+(a_3-m)v)^2}\\
    &+v(1-v)w\left[\frac{c_3m(a_3-e_3)}{(m+(a_3-m)v)^2}\left(\frac{c_1a_1v}{(a_1+w)^2}
    +\frac{c_2a_2(1-v)}{(a_2+w)^2}\right)\right.\\
    &+\left.\left(\frac{e_1c_1m}{(a_1+w)^2}-\frac{e_2c_2}{(a_2+w)^2}\right)\left(\frac{c_1w}{a_1+w}-\frac{c_2w}{a_2+w}\right)\right],\\
    \alpha_3(w)&=\frac{e_1e_3c_1c_3mm'xv(1-v)w}{(a_1+w)(m+(a_3-m)v)^2}.
\end{align*}

Notice that the sign of $\alpha_3(w)$ is totally determined by $m'(w)$. Let 
\begin{align}\label{H-func}
    H(w)=\alpha_1(w)\alpha_2(w)-\alpha_3(w),
\end{align}
then by Routh-Hurwitz criterion, $E^*$ is locally asymptotically stable if $\alpha_i>0$ and $H>0$ for $i=1,2,3$, and the we can obtain the following proposition.

\begin{proposition}\label{prop:las-positive}
    Along the positive equilibrium branch $\Gamma$, the following statements hold.
    \begin{itemize}
        \item[(a)] If $m'(w)<0$, then $E^*(w)$ is unstable.
        \item[(b)] If $m'(w)=0$, then the Jacobian matrix at $E^*(w)$ has a zero eigenvalue.
        \item[(c)] If $m'(w)>0$, then $E^*(w)$ is locally asymptotically stable if $\alpha_1(w)>0$ and $H(w)>0$.
    \end{itemize}
\end{proposition}




\begin{remark}
    Due to the parametrization and the expression of $\alpha_3$, we have a strong connection between the stability and $m'(w)$. \Cref{prop:las-positive} provides conditions for the local stability of the positive equilibrium branch. If $m$ is decreasing with respect to $w$, then it is clear that $E^*$ is unstable. However, while $m$ is increasing, it still depends on the sign of $\alpha_1$ and $H$, so we provide \Cref{thm:for-m-large} to ensure the stability when $m$ is large and present the numerical evidence of the complexity of the stability when $m$ is not sufficiently large in \Cref{subsec:numerical-hopf}. When $m'(w)=0$, it might be a fold on the branch, so \Cref{thm:saddle-node} provides sufficient conditions for the saddle-node bifurcation.
\end{remark}

\begin{theorem}\label{thm:for-m-large}
    Assume that \eqref{eq:generic-assumption-Delta-xz} and \eqref{eq:generic-assumption-I_z^y} hold, $0<\mcI_z^1<\frac{e_2}{a_2}$, $a_2\leq e_2$, and $\max\{c_1,c_2\}<1$. If $\mcI_z^y>0$, then there is $\overline m>0$ such that $E^*(m)$ is locally asymptotically stable for all $m>\overline m$.
\end{theorem}

The proof of \Cref{thm:for-m-large} is shown in \Cref{appendix:proofs} and is based on \Cref{prop:las-positive}.

We emphasize the nonmonotonicity of $m$ in a biological sense. When $\mcI_z^y>0$, sufficiently strong relative predation efficiency $m$ produces a stable coexistence state whose limiting population densities remain strictly positive. In contrast, when $\mcI_z^y<0$, the admissible interval is $w\in(w_{xz},w_{xy})$. Therefore, the coexistence steady state cannot persist for arbitrarily large $m$. 



At the end of this section, we study the saddle-node bifurcation. 
In our case, the only possible fold occurs at $w=w_{SN}$, which is a point that provides a zero root of the characteristic equation. 


\begin{theorem}\label{thm:saddle-node}
    Assume that \eqref{eq:generic-assumption-Delta-xz}, \eqref{eq:generic-assumption-I_z^y}, and \eqref{eq:nondegeneracy-saddle-node} hold, $0<\mcI_z^1<\frac{e_2}{a_2}$, $a_2\leq e_2$, and $\max\{c_1,c_2\}<1$. If $a_3<e_3$, $a_1\leq e_1m_{SN}$, $c_1\leq c_2$, and $a_1>a_2$, then the system \eqref{eq:main-model-1}--\eqref{eq:main-model-3} undergoes a saddle-node bifurcation at $m=m_{SN}$ whenever $w_{SN}$ is in the admissible interval. Moreover, there exist $\varepsilon>0$ and a neighborhood $U$ of $E^*(w_{SN})$ such that for $m\in(m_{SN},m_{SN}+\varepsilon)$, the system has exactly two positive equilibria in $U$; while for
     $m\in(m_{SN}-\varepsilon,m_{SN})$, no positive equilibrium in $U$.
\end{theorem}
\begin{proof} 
Since $\Phi_m$ is a diffeomorphism between $\mathbb{R}_{>0}^3$ and $\Omega$ for $m>0$, it suffices to establish the saddle-node bifurcation for the transformed system \eqref{eq:transformed-system}.

First, we prove that
\begin{align}\label{sn-sign}
 \alpha_1(w_{SN})>0,\quad \alpha_2(w_{SN})>0,\quad  \alpha_3(w_{SN})=0.
\end{align}
Since $m'(w_{SN})=0$, the expression of $\alpha_3$ gives $\alpha_3(w_{SN})=0$.
Using $x>0$, $a_3<e_3$, $a_1\leq e_1m_{SN}$, and $a_2\leq e_2$,
from the expression of $\alpha_1$ we see that $\alpha_1(w_{SN})>0$.
To prove $\alpha_2(w_{SN})>0$, set
\begin{align*}
J_1:=x(f_Y'+f_Z')+\frac{c_3m(a_3-e_3)w}{(m+(a_3-m)v)^2}\left(\frac{c_1a_1v}{(a_1+w)^2}+\frac{c_2a_2(1-v)}{(a_2+w)^2}\right).
\end{align*}
Using $f_Y'(v)+f_Z'(v)=c_3m(e_3-a_3)/[m+(a_3-m)v]^2$ and $A(x,v,w)=0$ at a positive equilibrium, we have 
  \begin{align*}
        J_1
        &=\frac{c_3m(e_3-a_3)}{(m+(a_3-m)v)^2}\left(1-\frac{c_1vw}{a_1+w}-\frac{c_2w(1-v)}{a_2+w}-\frac{c_1a_1vw}{(a_1+w)^2}-\frac{c_2a_2(1-v)w}{(a_2+w)^2}\right)\\
        &=\frac{c_3m(e_3-a_3)}{(m+(a_3-m)v)^2}\left(1-c_1v-c_2(1-v)+\frac{c_1va_1^2}{(a_1+w)^2}+\frac{c_2(1-v)a_2^2}{(a_2+w)^2}\right)\\
        &\geq\frac{c_3m(e_3-a_3)}{(m+(a_3-m)v)^2}(1-\max\{c_1,c_2\})>0.
    \end{align*}
 Next, set
 \[
 J_2:=\frac{e_1c_1m}{(a_1+w)^2}-\frac{e_2c_2}{(a_2+w)^2},\quad J_3:=\frac{c_1w}{a_1+w}-\frac{c_2w}{a_2+w}.
 \]
    By using \eqref{eq:relation-between-wSN-mSN} and $a_3<e_3$, we have
\begin{align*}
J_2=\frac{e_2c_2(a_3-e_3)}{e_3(a_2+w_{SN})^2}<0.
\end{align*}
Moreover, since $c_1\leq c_2$ and $c_1a_2\leq c_2a_2< c_2a_1$, we obtain 
\[
J_3=\frac{(c_1-c_2)w_{SN}+(c_1a_2-c_2a_1)}{(a_1+w_{SN})(a_2+w_{SN})}w_{SN}<0.
\]
Therefore, we have $J_2J_3>0$ at $w=w_{SN}$.
Now, using $m'(w_{SN})=0$ and the expression of $\alpha_2(w_{SN})$, we have
\[
\alpha_2(w_{SN})=v(1-v)J_1+xw\left(
        \frac{e_1c_1mv}{(a_1+w)^2}
        +\frac{e_2c_2(1-v)}{(a_2+w)^2}
        \right)
        +v(1-v)wJ_2J_3>0.
\]

    
    Consequently, \eqref{sn-sign} holds, and thus the characteristic polynomial at the fold is
    \[
        \lambda
        \left(
        \lambda^2+\alpha_1(w_{SN})\lambda+\alpha_2(w_{SN})
        \right).
    \]
    Hence, zero is a simple eigenvalue of $D\bm{\Tilde{F}}$ evaluated at the positive equilibrium $(x,v,w)$ with $m=m_{SN}$, while the remaining two
    eigenvalues have strictly negative real parts.


   We next verify the two nondegeneracy conditions for a saddle-node
    bifurcation.
    Let 
    \begin{align*}
      L(x,v,w):=\left(
        \frac{1-\frac{a_3}{e_3}}{x},\frac{1-v+\frac{a_3}{e_3}v}{v(1-v)},\frac{1-\frac{a_3}{e_3}}{w}
   \right)^\top  ,\quad  \ell:=L(x(w_{SN}),v(w_{SN}),w_{SN}).
    \end{align*}
    In $\{(x,v,w):x>0,v\in(0,1),w>0\}$, we have
    \begin{align*}
        L^\top\bm{\Tilde{F}}=\left(1-\frac{a_3}{e_3}\right)\frac{\Tilde{X}}{x}+\left(1-v+\frac{a_3}{e_3}v\right)\frac{\Tilde{V}}{v(1-v)}+\left(1-\frac{a_3}{e_3}\right)\frac{\Tilde{W}}{w}=B-\frac{a_3}{e_3}C=\frac{e_1c_1}{a_1+w}(m-m(w)).
    \end{align*}
    Since $m'(w_{SN})=0$ and $m_{SN}=m(w_{SN})$, we have
    \begin{align*}
        \ell^\top D\bm{\Tilde{F}}(x,v,w)\bigg|_{(x(w_{SN}),v(w_{SN}),w_{SN})}&=\left[\nabla(L^\top\bm{\Tilde{F}})\right]^\top\bigg|_{(x(w_{SN}),v(w_{SN}),w_{SN})}\\
         &=\Big(
            \frac{\partial}{\partial x}\big(B-\frac{a_3}{e_3}C\big),\frac{\partial}{\partial v}\big(B-\frac{a_3}{e_3}C\big),\frac{\partial}{\partial w}\big(B-\frac{a_3}{e_3}C\big)
        \Big)\bigg|_{w=w_{SN}}\\
        &=\Big(
            0,0,\frac{\partial}{\partial w}\big(\frac{e_1c_1}{a_1+w}(m-m(w))\big)\bigg|_{w=w_{SN}}
        \Big)=(0,0,0).
    \end{align*}
   Hence, $\ell$ is a left eigenvector associated with the zero eigenvalue.

    Along the positive equilibrium branch, write 
    \[
    E^*(w):=(x(w),v(w),w)^\top,\quad r:=(E^*)'(w_{SN})=(x'(w_{SN}),v'(w_{SN}),1)^\top.
    \]
    Differentiating $\bm{\Tilde{F}}(E^*(w);m(w))=0$ in $w$ evaluated at $w=w_{SN}$, and using 
   $m'(w_{SN})=0$,
   we obtain
    \begin{align*}
       0=D\bm{\Tilde{F}}(E^*(w_{SN});m_{SN})r+\bm{\Tilde{F}}_m(E^*(w_{SN});m_{SN})m'(w_{SN})=D\bm{\Tilde{F}}(E^*(w_{SN});m_{SN})r.
    \end{align*}
    Hence, $r$ is a right eigenvector associated with the zero eigenvalue.

    The normal form \cite{perko2013differential} of \eqref{eq:transformed-system} at $(E^*,m_{SN})$ can be written as
    \begin{align*}
        \dot\eta=a(m-m_{SN})+b\eta^2+O\left(
        |\eta|^3+|(m-m_{SN})\eta|+(m-m_{SN})^2
        \right),
    \end{align*}
    where
    \begin{align*}
        a&=\frac{\ell^\top\bm{\Tilde{F}}_m(E^*(w_{SN});m_{SN})}{\ell^\top r}=\frac{1}{\ell^\top r}\frac{e_1c_1}{a_1+w_{SN}}\neq0\quad\textup{and}\\
        b&=\frac{1}{2}\frac{\ell^\top D^2\bm{\Tilde{F}}(E^*(w_{SN});m_{SN})(r,r)}{\ell^\top r}=-\frac{1}{2}\frac{\ell^\top\bm{\Tilde{F}}_m(E^*(w_{SN});m_{SN})m''(w_{SN})}{\ell^\top r}\neq0
    \end{align*}
    which is obtained by differentiating $\bm{\Tilde{F}}(E^*(w);m(w)))=0$ twice.
    
    In addition, note that $ab=-a^2m''(w_{SN})/2<0$. Hence, for $m<m_{SN}$ sufficiently close to $m_{SN}$, no positive equilibrium exists near $E^*(w_{SN})$; while $m>m_{SN}$ sufficiently close to $m_{SN}$, there exist exactly two positive equilibria.
    Thus, the proof is complete.
\end{proof}

\begin{remark}
    Under the hypothesis of \Cref{thm:saddle-node}, if we parametrize $E^*$ by $w$, then when $w$ is near $w_{SN}$, the following hold.
    \begin{enumerate}
        \item $E^*(w)$ is a saddle whenever $w<w_{SN}$.
        \item $E^*(w)$ is locally asymptotically stable whenever $w>w_{SN}$.
    \end{enumerate}
\end{remark}

\section{Numerical Results}\label{sec:numerical}

According to the complexity of $H(w)$, the local stability is difficult to analyze. In this section, we present several numerical results for the complex dynamics of \eqref{eq:main-model-1}--\eqref{eq:main-model-3}, including multiple attractors, the Hopf bifurcation, and chaos, using the software Auto07p and MATLAB Online. 

\subsection{Multiple Attractors}

We use MATLAB online to demonstrate the multiple attractors.
First, we consider the alternative extinction. \Cref{lemma:consumer-extinction} and \Cref{lemma:y-extinct} illustrate the extinction of consumers, and IG prey can still depend on the density of the species. \Cref{fig:bistable-E0-E1} 
illustrates two possible long-term outcomes for strictly positive initial data: convergence to $E_0$ and $E_1$ with $\mcI_y^1=-0.75<0$ and $\mcI_z^1=-0.0909<0$. \Cref{fig:bistable-E0-Exz} shows the convergence to $E_0$ and $E_{xz}$ with $a_1-a_2=0.5>0$, $\frac{a_2}{d_2}\mcI_z^1=0.0909$, and $\frac{a_1}{d_1}\mcI_y^1=0.0125$. 

\begin{figure}
    \centering
    \begin{subfigure}[t]{.45\textwidth}
        \begin{subfigure}{.49\textwidth}
            \includegraphics[width=\linewidth]{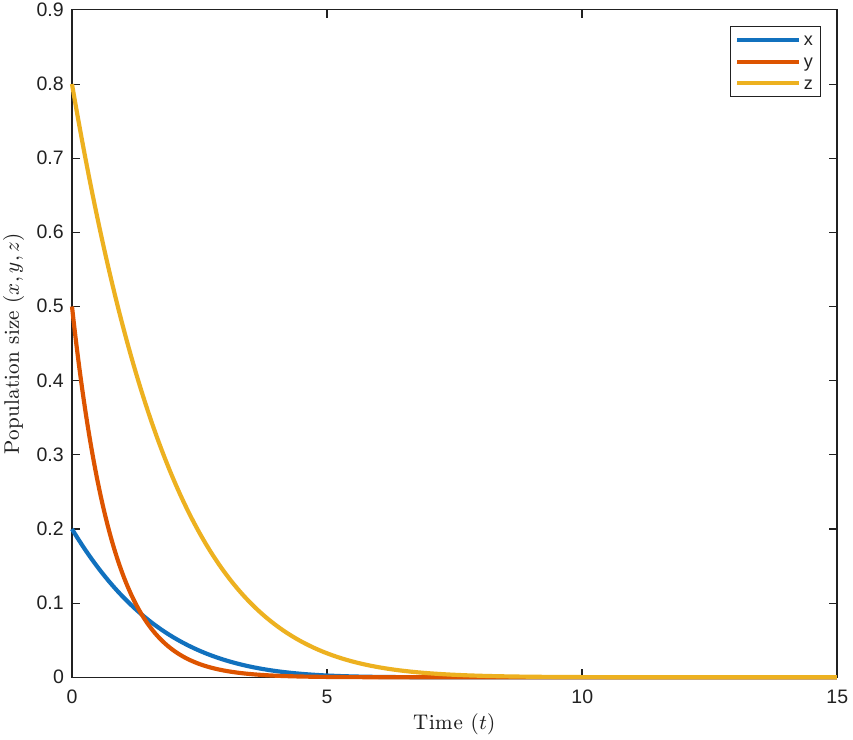}
        \end{subfigure}\hfill
        \begin{subfigure}{.49\textwidth}
            \includegraphics[width=\linewidth]{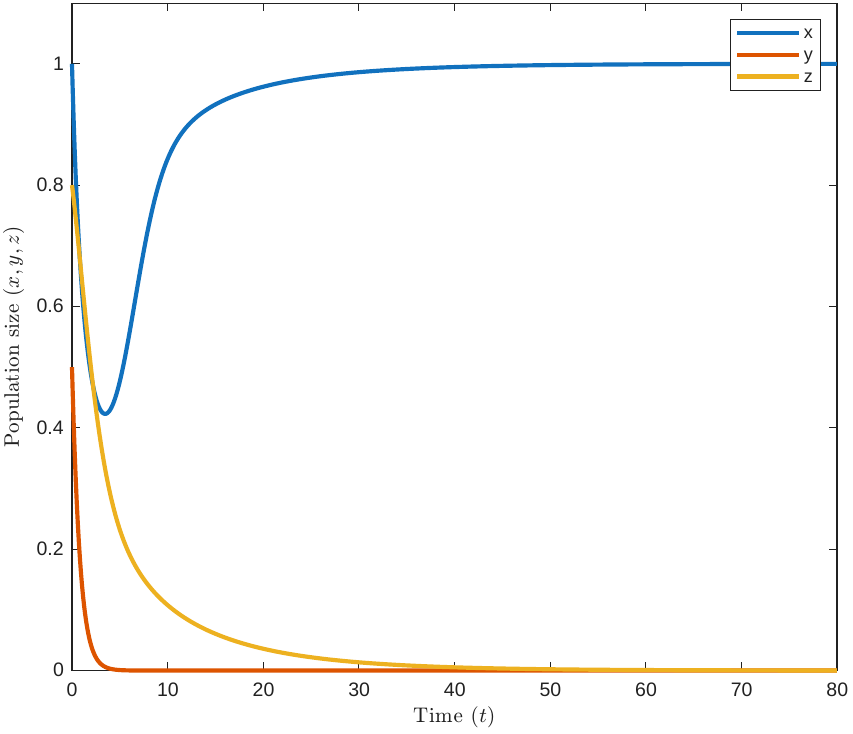}
        \end{subfigure}
        \caption{$c_1=0.5$, $c_2=3$, $c_3=0.5$, $e_1=1$, $e_2=1$, $e_3=1$, $a_1=2$, $a_2=3.3$, $a_3=1$, $d_1=1$, $d_2=1$, $m=1$. Initial data are $(0.2,0.5,0.8)$ and $(1,0.5,0.8)$.}
        \label{fig:bistable-E0-E1}
    \end{subfigure}\hfill
    \begin{subfigure}[t]{.45\textwidth}
        \begin{subfigure}{.49\textwidth}
            \includegraphics[width=\textwidth]{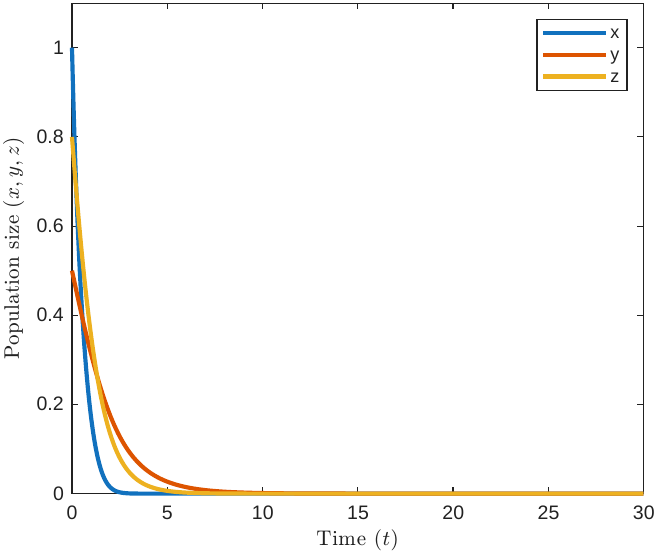}
        \end{subfigure}
        \begin{subfigure}[t]{.49\textwidth}
            \includegraphics[width=\textwidth]{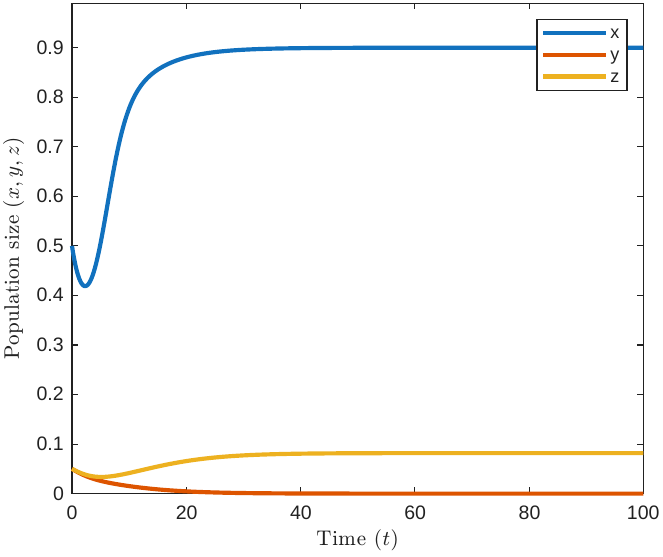}
        \end{subfigure}
        \caption{$c_1=6.75$, $c_2=0.65$, $c_3=0.10$, $e_1=0.09$, $e_2=1$, $e_3=1$, $a_1=1$, $a_2=0.5$, $a_3=1$, $d_1=0.6$, $d_2=1.1$, $m=1$. Initial data are $(1,0.5,0.8)$ and $(0.50,0.05,0.05)$.}
        \label{fig:bistable-E0-Exz}
    \end{subfigure}
    \caption{The figures on the left suggest that $E_0$ and $E_1$ are attractors. The figures on the right suggest that $E_0$ and $E_{xz}$ are attractors.}
    \label{fig:bistability-of-extinction}
\end{figure}

The next case is the numerical result for the different types of coexistence. \Cref{fig:bistable-positive-eq-positive-periodic} demonstrates bistability between a positive equilibrium and a positive periodic solution. Therefore, the type of coexistence of \eqref{eq:main-model-1}--\eqref{eq:main-model-3} might depend on the initial data.

\begin{figure}\centering
    \begin{subfigure}{.35\textwidth}
        \includegraphics[width=\textwidth]{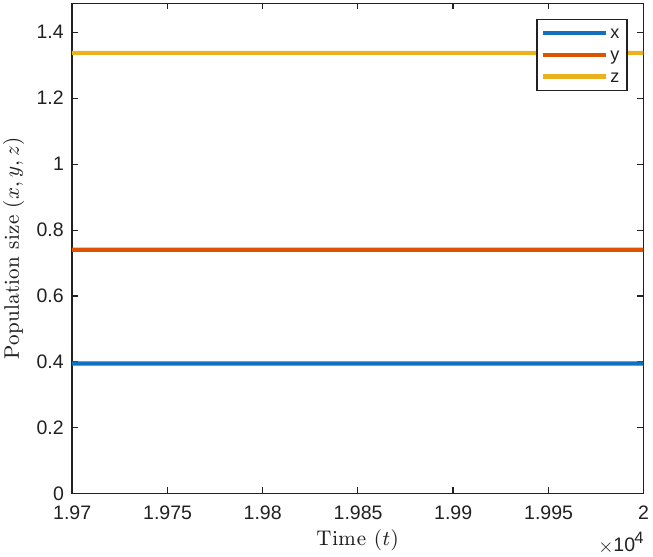}
    \end{subfigure}\hspace{1cm}
    \begin{subfigure}[t]{.35\textwidth}
        \includegraphics[width=\textwidth]{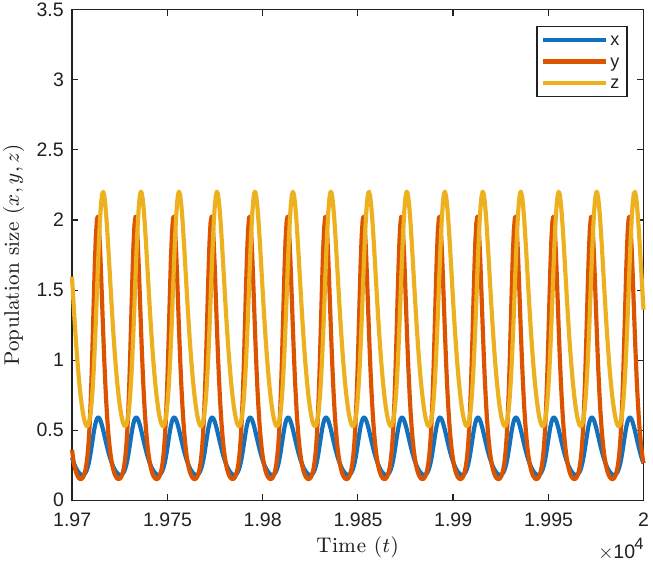}
    \end{subfigure}
    \caption{The figures suggest that both positive equilibrium and positive periodic solution are stable. The parameters are chosen by $c_1=0.09$, $c_2=1.17$, $c_3=1.98$, $e_1=263.00$, $e_2=0.01$, $e_3=0.32$, $a_1=8.87$, $a_2=0.04$, $a_3=0.47$, $d_1=1.13$, $d_2=0.28$, $m=1.78$. The initial data are chosen by $(0.39,0.74,1.34)$ and $(0.43,0.78,0.45)$.}
    \label{fig:bistable-positive-eq-positive-periodic}
\end{figure}

\subsection{Hopf Bifurcation}\label{subsec:numerical-hopf}

In an intraguild predation system, the additional trophic link may either stabilize or destabilize the positive equilibrium, depending on $m$.
In \Cref{sec:coexistence}, we have classified the structure of $\Gamma$ into eight cases by using the quantities $\mcI_z^y$, $\Delta_{xy}$, and $\Delta_{xz}$, and provide the results when $m$ is large and small. However, these results can not determine the stability of the positive equilibrium. We therefore apply the Hopf-bifurcation theorem \cite{liu1994criterion} to search for the parameters, use AUTO-07p to continue the positive equilibrium, and present the results of the Hopf bifurcation for each case. 

\Cref{fig:Hopf-bifurcation} presents numerically the possibility of the existence of a periodic solution that bifurcates from the positive equilibrium. The cases with $\mcI_z^y>0$ are displayed in \Cref{fig:Hopf-bifurcation}, while those with $\mcI_z^y<0$ are displayed in \Cref{fig:Hopf-bifurcation}. The solid and dashed lines represent the stable and unstable branches, respectively. The black line represents the positive equilibrium that bifurcates from $E_{xz}$, and the green lines represent the maximum and the minimum of the periodic solution. The used parameters are as follows.

\begin{itemize}
    \item \Cref{fig:hopf-1}: $c_1=0.83$, $c_2=0.16$, $c_3=2.95$, $e_1=14.20$, $e_2=0.69$, $e_3=0.25$, $a_1=0.23$, $a_2=0.07$, $a_3=1.77$, $d_1=0.38$, $d_2=0.09$
    \item \Cref{fig:hopf-2}: $c_1=0.48$, $c_2=0.09$, $c_3=3.85$, $e_1=8.52$, $e_2=1.05$, $e_3=1.55$, $a_1=0.25$, $a_2=0.07$, $a_3=4.03$, $d_1=0.04$, $d_2=1.03$
    \item \Cref{fig:hopf-5}: $c_1=0.49$, $c_2=0.64$, $c_3=0.55$, $e_1=13.30$, $e_2=0.68$, $e_3=0.34$, $a_1=0.20$, $a_2=0.17$, $a_3=6.53$, $d_1=0.32$, $d_2=0.03$
    \item \Cref{fig:hopf-6}: $c_1=0.60$, $c_2=0.44$, $c_3=11.79$, $e_1=12.58$, $e_2=2.98$, $e_3=0.30$, $a_1=1.19$, $a_2=0.65$, $a_3=2.29$, $d_1=8.42$, $d_2=1.61$
\end{itemize}

\begin{figure}
    \centering
    \begin{subfigure}[t]{.24\textwidth}
        \includegraphics[width=\linewidth]{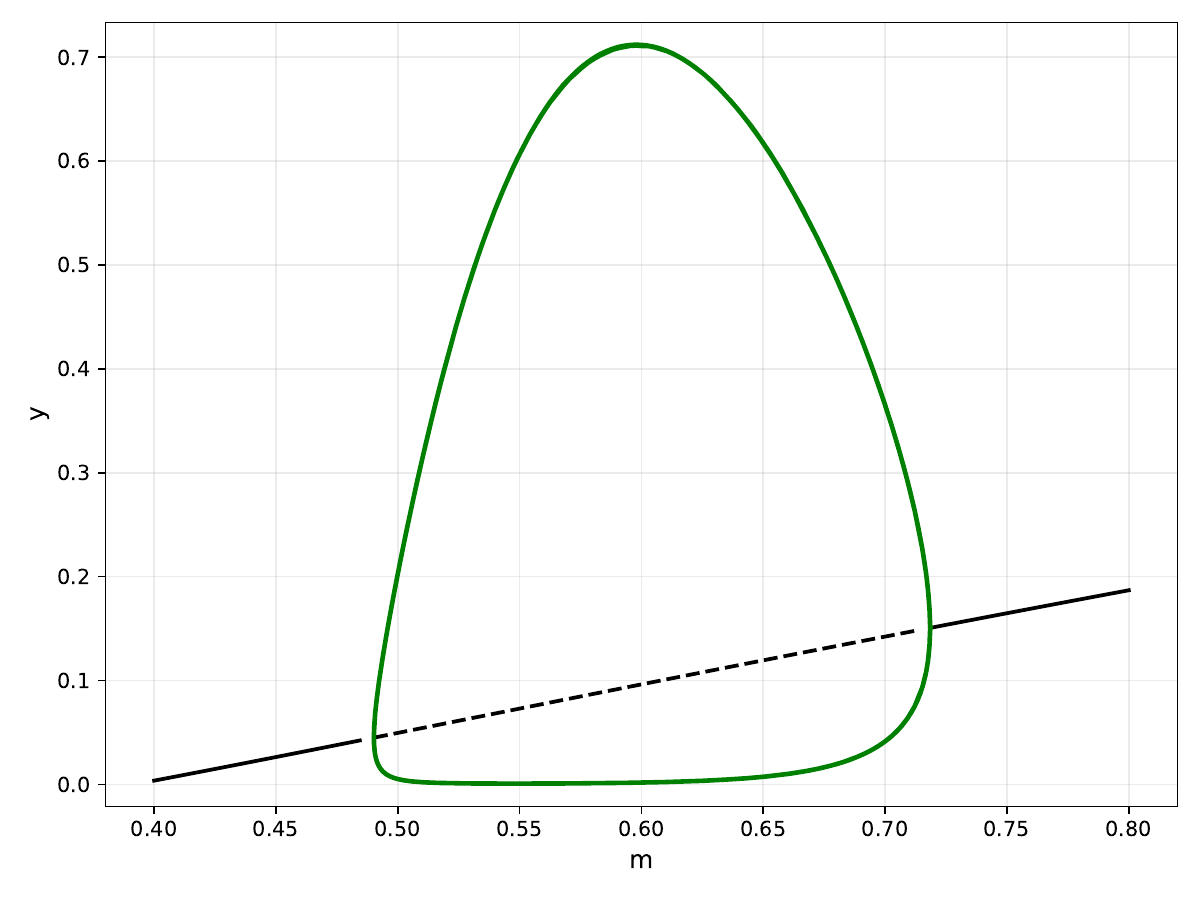}
        \caption{Case (i)}
        \label{fig:hopf-1}
    \end{subfigure}\hfill
    \begin{subfigure}[t]{.24\textwidth}
        \includegraphics[width=\linewidth]{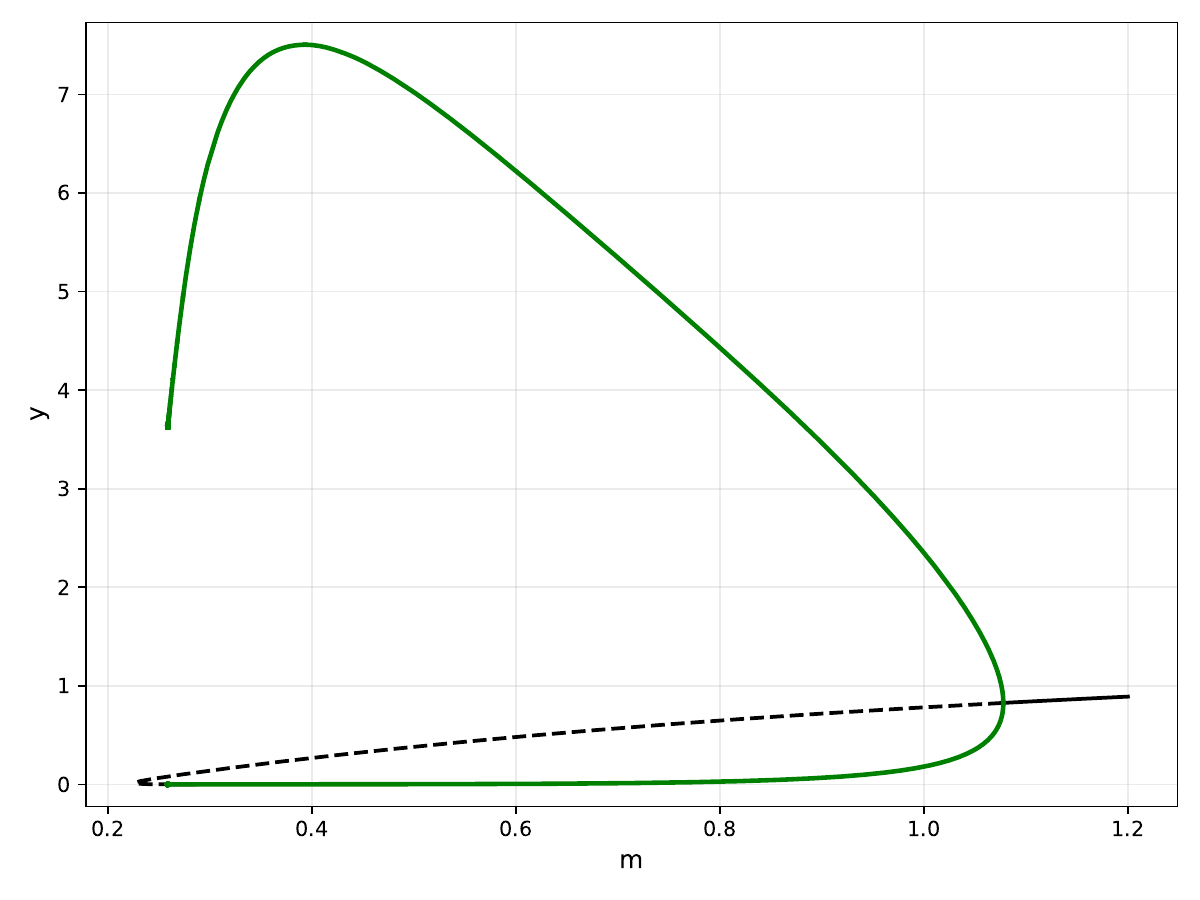}
        \caption{Case (ii)}
        \label{fig:hopf-2}
    \end{subfigure}\hfill
    \begin{subfigure}[t]{.24\textwidth}
        \includegraphics[width=\linewidth]{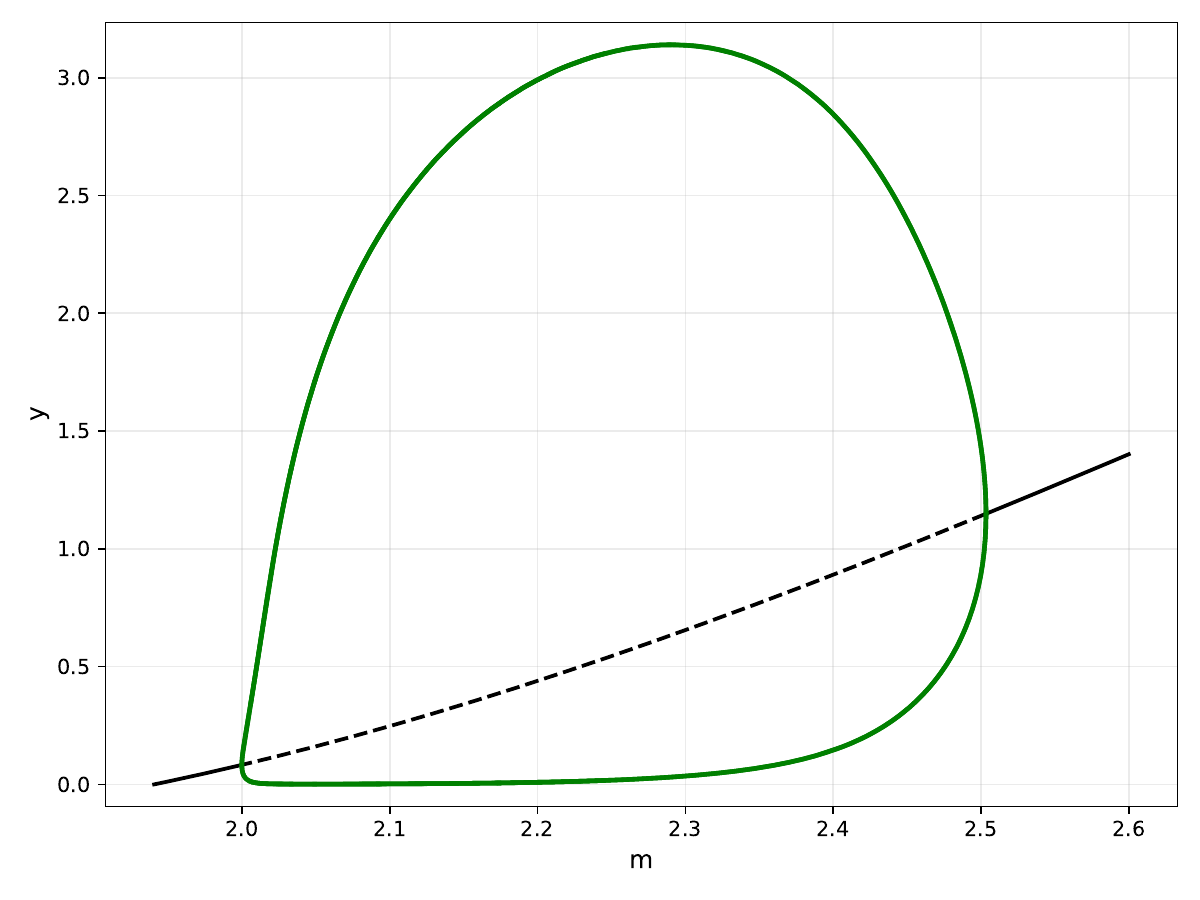}
        \caption{Case (v)}
        \label{fig:hopf-5}
    \end{subfigure}\hfill
    \begin{subfigure}[t]{.24\textwidth}
        \includegraphics[width=\linewidth]{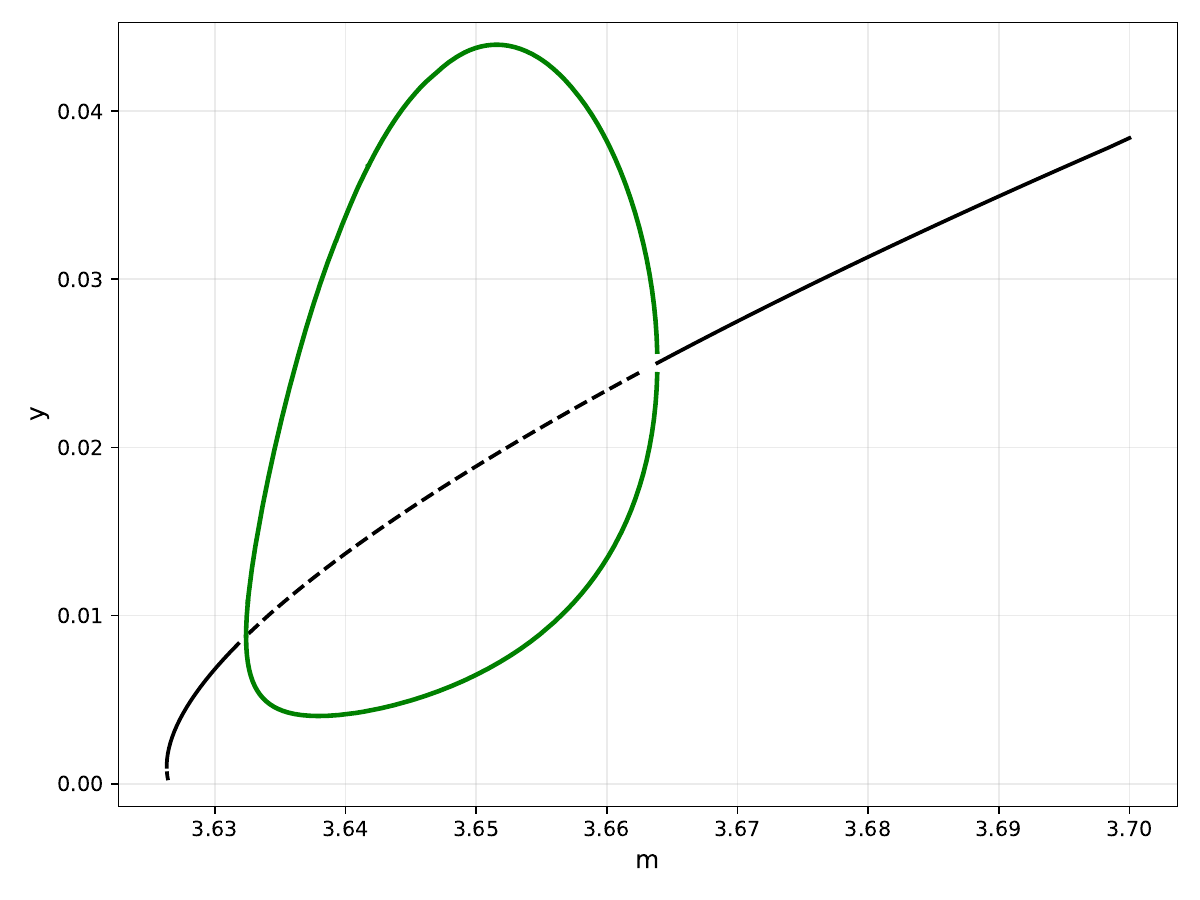}
        \caption{Case (vi)}
        \label{fig:hopf-6}
    \end{subfigure}
    \caption{Numerical results of Hopf bifurcation of case (i), (ii), (v), and (vi) in \Cref{tab:classification-bifurcation-and-coexistence}.}
    \label{fig:Hopf-bifurcation}
\end{figure}

These results demonstrate that the relative predation efficiency $m$ affects not only the existence of the positive equilibrium but also the type of coexistence dynamics; that is, increasing $m$ may stabilize or destabilize the positive equilibrium and create or destroy a periodic orbit.

\subsection{Chaos}

We next investigate whether the oscillatory coexistence dynamics can develop chaotic behavior. The chaotic behavior can also be found in the Lotka-Volterra IGP model \cite{tanabe2005omnivory}. In our model, we also found parameters for which chaotic behavior occurs, and the numerical results are shown in \Cref{fig:chaos}. For each value of $m$, the system is first solved over the time interval $[0,5000]$ to ensure the solution is close to the strange attractor. The state $(x(5000),y(5000),z(5000))$ is then used to investigate the long-term dynamics.

\begin{figure}
    \centering
    \begin{subfigure}[t]{.35\textwidth}
        \includegraphics[width=\textwidth]{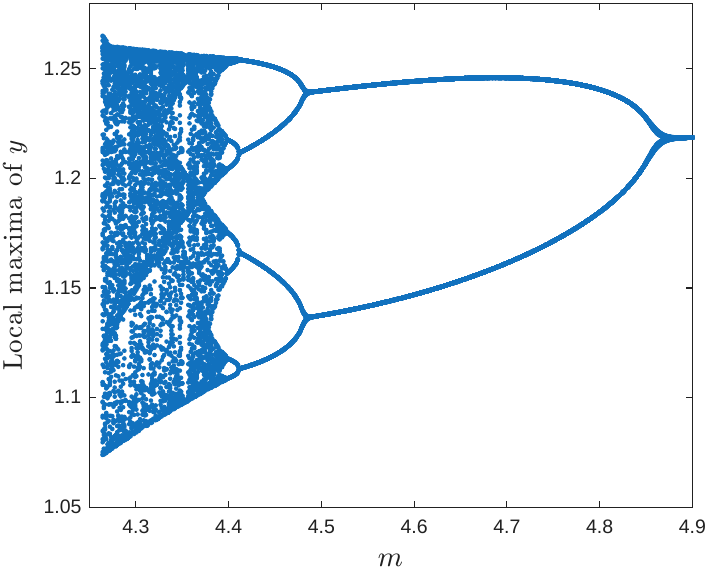}
        \caption{The bifurcation diagram}
        \label{fig:chaos-bifurcation-diagram}
    \end{subfigure}\hspace{1cm}
    \begin{subfigure}[t]{.37\textwidth}
        \includegraphics[width=\textwidth]{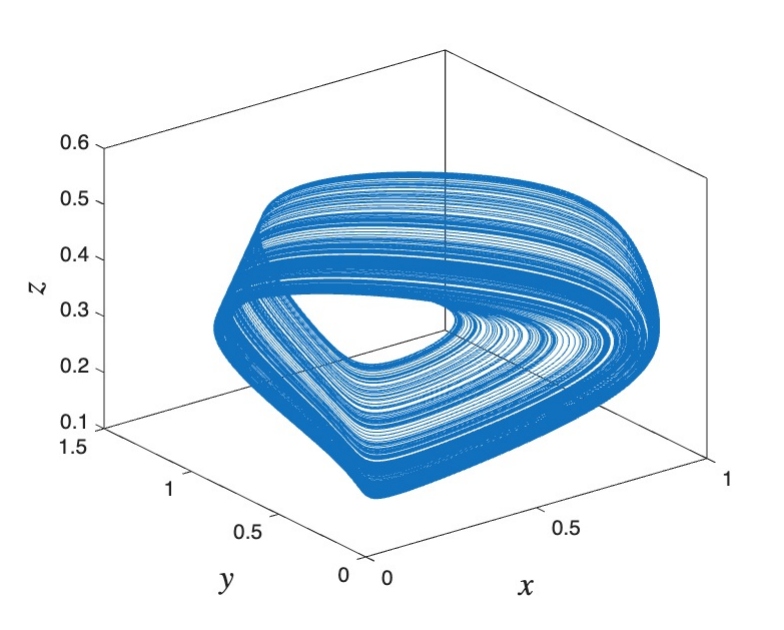}
        \caption{The strange attractor}
        \label{fig:chaos-strange-attractor}
    \end{subfigure}
    \caption{These two figures use parameters $c_1=1.5985$, $c_2=0.1333$, $c_3=0.4511$, $e_1=1.2604$, $e_2=1.2577$, $e_3=0.1168$, $a_1=12.1011$, $a_2=0.4320$, $a_3=1.4951$, $d_1=0.1885$, $d_2=0.0545$, and initial data $(0.3362,0.2363,0.4810)$.}
    \label{fig:chaos}
\end{figure}

\Cref{fig:chaos-bifurcation-diagram} presents the bifurcation diagram obtained by recording the local maxima of $y(t)$ for $t>5000$ with the $m\in[4.3,4.9]$. To distinguish chaotic dynamics from a long-period periodic orbit, we compute the maximal Lyapunov exponent. The initial data is chosen by $(0.3362,0.2363,0.4810)^\top$, and the initial perturbation direction is chosen by $(1,1,1)^{\top}$. For $m=4.31$, the resulting maximal Lyapunov exponent is approximately $0.0099$. The strictly positive maximal Lyapunov exponent provides numerical evidence for chaotic dynamics. \Cref{fig:chaos-strange-attractor} displays the strange attractor in $\mathbb{R}^3_{>0}$ when $m=4.31$. The trajectory remains bounded and exhibits a complicated, aperiodic structure in phase space. 

\section{Conclusion}

In this work, we studied a three-species intraguild predation model with ratio-dependent trophic functions.
We derived explicit conditions for consumer extinction, total extinction, and competitive exclusion.
We also examined how these boundary outcomes relate to positive equilibria 
by treating $m$ as a bifurcation parameter and obtained a global description of the entire positive equilibrium set.

Our main contribution is the explicit global parameterization of all positive equilibria in terms of the weighted consumer-to-resource ratio $w$. Under the generic assumptions, the positive equilibrium set forms a single smooth curve whose closure contains $(E_{xz},m_{xz})$
and its global continuation is completely determined by the three quantities $\mathcal I_z^y$, $\Delta_{xy}$, and $\Delta_{xz}$. Their signs divide the parameter space into eight generic regimes, which give rise to six distinct global geometries of the equilibrium branch.

When $\mathcal I_z^y>0$, the branch either increases monotonically toward arbitrarily large values of $m$ or develops a unique minimum fold before extending to $m=\infty$. When $\mathcal I_z^y<0$, the branch connects $(E_{xz},m_{xz})$ to $(E_{xy},m_{xy})$ and may be monotonically increasing, monotonically decreasing, or possess a unique minimum or maximum fold. 
In the fold cases, two positive equilibria may coexist at the same value of $m$. Thus, the eight parameter regimes distinguish the different mechanisms determining the branch, while the six geometries offer a complete classification of its possible global configurations.

This classification has a biological interpretation.
The sign of $\mathcal I_z^y>0$ determines whether a rare IG predator has positive per-capita growth when its gain comes solely from the IG prey. When $\mathcal I_z^y>0$, this limiting growth rate is positive, and the positive-equilibrium branch extends to arbitrarily large values of $m$.
However, when $\mathcal{I}_z^y<0$, the IG prey cannot alone support
a rare IG predator, and the positive equilibrium curve instead terminates at the IG-predator-free equilibrium $E_{xy}$. Consequently, $\mathcal I_z^y$ determines the global endpoint of the coexistence curve, whereas $\Delta_{xy}$ and $\Delta_{xz}$ determine its orientation and possible fold structure.

We remark that the geometry of the positive equilibrium branch does not by itself completely determine stability. Our determinant relation proves that decreasing portions of the branch are unstable, while the Routh--Hurwitz conditions provide stability criteria along increasing portions of the branch. This shows that changing $m$ may cause a positive equilibrium to gain or lose stability, and may generate periodic coexistence. The numerical bifurcation diagrams further illustrate how stable coexistence equilibria, periodic solutions, multistability, and, in some parameter regimes, chaotic dynamics can be observed around the six classified branch geometries.

Overall, our results show that the effect of $m$ is not simply beneficial or harmful to coexistence, but depends on the asymmetric trophic interactions in the system. Thus, it should be viewed not merely as a stabilizing mechanism, but as a factor that can qualitatively reorganize the dynamics of the entire IGP community with ratio-dependent trophic function. 

\appendix

\section{Proofs of Selected Results}\label{appendix:proofs}

\begin{proof}[Proof of \Cref{lemma:y-extinct}]
    Whenever $y_0=0$, then $y(t)=0$ for all $t>0$, Hence, it suffices to consider $y_0>0$. Inspire by Hsu et al. \cite{hsu2001rich}, let $V=\frac{y^\theta}{z}$, where $\theta>0$ will be determined later. Let $\omega=my+z$, then 
    \begin{align*}
        \dot{V}(t)=\frac{y^\theta}{z}\left(\frac{P(x,\omega)}{(a_1x+\omega)(a_2x+\omega)}-\frac{{\theta c_3z}+e_3c_3y}{a_3y+z}\right),
    \end{align*}
    where 
    \begin{align*}
        P(x,\omega;\theta)&=[a_2\theta(e_1c_1m-a_1d_1)-a_1(e_2c_2-a_2d_2)]x^2+(d_2-\theta d_1)\omega^2\\
        &\quad+[\theta(e_1c_1m-a_1d_1-a_2d_1)-(e_2c_2-a_2d_2-a_1d_2)]x\omega\\
        &=:A(\theta)x^2+B(\theta)x\omega+C(\theta)\omega^2
    \end{align*}
    Take $\theta=\theta^*=(e_2c_2-a_2d_2+a_1d_2)/(e_1c_1m-a_1d_1+a_2d_1)>0$. The discriminant of $P$ is
    \begin{align*}
        \Delta(\theta^*)=4\theta^*d_1d_2(a_1-a_2)\left(\frac{a_1}{d_1}\mcI_y^1-\frac{a_2}{d_2}\mcI_z^1\right)\leq0
    \end{align*}
    and
  \begin{align*}
        A(\theta^*)&=\frac{1}{e_1c_1m-a_1d_1+a_2d_1}\left[a_1a_2\mcI_y^1\mcI_z^1(a_2-a_1)+a_1a_2d_1d_2\left(\frac{a_1}{d_1}\mcI_y^1-\frac{a_2}{d_2}\mcI_z^1\right)\right]<0\\
        B(\theta^*)&=\frac{2d_1d_2}{e_1c_1m-a_1d_1+a_2d_1}\left(\frac{a_1^2}{d_1}\mcI_y^1-\frac{a_2^2}{d_2}\mcI_z^1\right).\\
        C(\theta^*)&=\frac{d_1d_2}{e_1c_1m-a_1d_1+a_2d_1}\left[\left(\frac{a_1}{d_1}\mcI_y^1-\frac{a_2}{d_2}\mcI_z^1\right)+(a_2-a_1)\right]<0.
    \end{align*}
    Apply the discussion in \cite{hsu2001rich}, we have
        $P(x,\omega;\theta^*)\leq0$ for all $x,\omega\geq0$.
    Finally, if $\dot{V}=0$, then $y=0$; if $\dot{V}<0$, from
    \begin{align*}
        \dot{V}=V\left(\frac{P(x,\omega)}{(a_1x+\omega)(a_2x+\omega)}-\frac{{\theta^* c_3z}+e_3c_3y}{a_3y+z}\right)\leq-V\cdot\frac{e_3c_3y+\theta^*c_3z}{a_3y+z}\leq-\Tilde{C}V,
    \end{align*}
    where $\Tilde{C}=\min\left\{\frac{e_3c_3}{a_3},\theta^* c_3\right\}>0$, then we have
     $y(t)^{\theta^*}\leq z(t)V(0)e^{-\Tilde{C}t}\leq M_zV(0)e^{-\Tilde{C}t}\to0$ as $t\to\infty$.
    Therefore, $\lim\limits_{t\to\infty}y(t)=0$.
\end{proof}

\begin{proof}[Proof of \Cref{thm:for-m-large}]
    First, notice that as $w\to\infty$, we have $m\to\infty$, and then
    \begin{align*}
        v&=\frac{(d_2-\frac{e_2c_2}{a_2+w})m}{e_3c_3-(d_2-\frac{e_2c_2}{a_2+w})(a_3-m)}=\frac{(d_2-\frac{e_2c_2}{a_2+w})}{\frac{e_3c_3}{a_3-m}-(d_2-\frac{e_2c_2}{a_2+w})}\frac{m}{a_3-m}\to1,\\
        x&=1-\frac{c_1vw}{a_1+w}-\frac{c_2w(1-v)}{a_2+w}\to1-c_1>0,\\
        m'&=\frac{(a_3\mcI_z^y+e_3d_1)(a_2+w)^2+e_2c_2a_3(a_2-a_1)}{e_1e_3c_1(a_2+w)^2}\to\frac{a_3\mcI_z^y+e_3d_1}{e_1e_3c_1}.
        \end{align*}
    In addition, for $\alpha_i$, $i=1,2,3$, as  $w\to\infty$,
    \begin{align*}  
        \alpha_1&=x+\frac{e_3-a_3}{e_3c_3}\left(\frac{e_1c_1m}{a_1+w}-d_1\right)\left(d_2-\frac{e_2c_2}{a_2+w}\right)+\frac{c_1(e_1m-a_1)vw}{(a_1+w)^2}+\frac{c_2(e_2-a_2)(1-v)w}{(a_2+w)^2}\\
        &\to1-c_1+\frac{1}{e_3^2c_3}(a_3^2(\mcI_z^y)^2+e_3a_3d_2\mcI_z^y+e_3^2c_3d_1)>0,\\
        \alpha_2&=\frac{e_3-a_3}{e_3c_3}\left(\frac{e_1c_1m}{a_1+w}-d_1\right)\left(d_2-\frac{e_2c_2}{a_2+w}\right)\left(x-\frac{c_1a_1vw}{(a_1+w)^2}-\frac{c_2a_2(1-v)w}{(a_2+w)^2}\right)\\
        &\quad+x\left(\frac{e_1c_1mvw}{(a_1+w)^2}+\frac{e_2c_2(1-v)w}{(a_2+w)^2}\right)\\
        &\quad+v(1-v)w\left(\frac{e_1c_1m}{(a_1+w)^2}-\frac{e_2c_2}{(a_2+w)^2}\right)\left(\frac{c_1w}{a_1+w}-\frac{c_2w}{a_2+w}\right)\\
        &\quad+\frac{e_1c_1wm'(w)}{c_3(a_1+w)}\left(\frac{e_1c_1m}{a_1+w}-d_1\right)\left(d_2-\frac{e_2c_2}{a_2+w}\right)\\
        &\to\frac{1-c_1}{e_3^2c_3}(a_3^2(\mcI_z^y)^2+e_3a_3d_2\mcI_z^y+e_3^2c_3d_1)+\frac{d_2}{e_3^2c_3}a_3\mcI_z^y(a_3\mcI_z^y+e_3d_1)>0,\\
        \alpha_3&=\frac{e_1c_1xwm'(w)}{c_3(a_1+w)}\left(\frac{e_1c_1m}{a_1+w}-d_1\right)\left(d_2-\frac{e_2c_2}{a_2+w}\right)\to\frac{d_2}{e_3^2c_3}a_3\mcI_z^y(a_3\mcI_z^y+e_3d_1)(1-c_1)>0.
 \end{align*}
 Therefore, as $w\to\infty$,
        \begin{align*}
        H&\to\frac{1}{e_3^2c_3}(a_3^2(\mcI_z^y)^2+e_3a_3d_2\mcI_z^y+e_3^2c_3d_1)\left[(1-c_1)^2\right.\\
        &\quad+\left.\frac{1-c_1}{e_3^2c_3}(a_3^2(\mcI_z^y)^2+e_3a_3d_2\mcI_z^y+e_3^2c_3d_1)+\frac{d_2}{e_3^2c_3}a_3\mcI_z^y(a_3\mcI_z^y+e_3d_1)\right]>0
    \end{align*}
    Therefore, by the Routh-Hurwitz criterion, there is $\overline{m}>0$ such that if $m>\overline{m}$, the positive equilibrium is locally asymptotically stable.
\end{proof}


\section*{Acknowledgements}

CHW is supported partly by the National Science and Technology Council of Taiwan (NSTC 113-2628-M-A49-004-MY4) and the National Center for Theoretical Sciences.

\bibliographystyle{siam}
\bibliography{reference}
\end{document}